\documentclass[12pt,a4paper]{article}

\usepackage[T1]{fontenc}
\usepackage[utf8]{inputenc}
\usepackage{float}
\usepackage{graphicx}
\usepackage{subcaption} 
\usepackage[thinlines]{easytable}
\usepackage{hyperref}
\usepackage{wrapfig}
\usepackage{dingbat}
\usepackage{color}
\usepackage{refcount}
\usepackage[table]{xcolor}
\usepackage[toc,page]{appendix}
\usepackage{bbm}
\usepackage{enumitem}
\usepackage{amsmath,amssymb,amsthm}
\usepackage{mathtools}
\usepackage[english]{babel}
\usepackage{natbib}
\usepackage{xcolor}
\usepackage{bibunits}

\renewcommand{\baselinestretch}{1.2} % Increased line spacing (slightly)

\theoremstyle{definition}
\newtheorem{theo}{Theorem}[section]
\newtheorem{lem}[theo]{Lemma}
\newtheorem{prop}[theo]{Proposition}

\newtheorem{example}[theo]{Example}
\newtheorem{algorithm}[theo]{Algorithm}

\newtheorem{assumption}[theo]{Assumption}

\numberwithin{equation}{section}

\newcommand{\lb}{\left(}
\newcommand{\rb}{\right)}

\newcommand{\E}{\mathbb{E}}

\newcommand{\PR}{\mathbb{P}}

\usepackage{algorithm}
\usepackage{algpseudocode}

\newcommand\T{\top}
\usepackage{booktabs}
\usepackage{multirow}
\usepackage{nth}

\begin{document}
		
\title{
Multiscale Change Point Detection for\\
Functional Time Series
}

\author{
Tim Kutta\thanks{Department of Mathematics, Aarhus University, \texttt{tim.kutta@math.au.dk}; corresponding author} \and
Holger Dette\thanks{Fakultät für Mathematik, Ruhr-Universität Bochum, \texttt{holger.dette@rub.de}} \and
Shixuan Wang\thanks{Department of Mathematics and Statistics, University of Reading, \texttt{s.wang@reading.ac.uk}}
}

\maketitle

\begin{abstract}
\noindent We study the problem of detecting and localizing multiple changes in the mean parameter of a Banach space–valued time series. The goal is to construct a collection of narrow confidence intervals, each containing at least one (or exactly one) change, with globally controlled error probability. Our approach relies on a new class of weighted scan statistics, called \textit{Hölder-type statistics}, which allow a smooth trade-off between efficiency (enabling the detection of closely spaced, small changes) and robustness (against heavier tails and stronger dependence). For Gaussian noise, maximum weighting can be applied, leading to a generalization of optimality results known for scalar, independent data. Even for scalar time series, our approach is advantageous, as it accommodates broad classes of dependency structures and non-stationarity. Its primary advantage, however, lies in its applicability to functional time series, where few methods exist and established procedures impose strong restrictions on the spacing and magnitude of changes. We obtain general results by employing new Gaussian approximations for the partial sum process in Hölder spaces. As an application of our general theory, we consider the detection of distributional changes in a data panel. The finite-sample properties and applications to financial datasets further highlight the merits of our method.
\end{abstract}


\defaultbibliographystyle{apalike}
\defaultbibliography{Literature}

\begin{bibunit}
\section{Introduction}

Throughout this paper $(X_n)_{n \in \mathbb{N}}$ denotes a  time series 
living on a separable, real Banach space $(\mathbb{B}, \| \cdot \|)$. 
We consider the problem of detecting and localizing changes in the mean parameter of $(X_n)_{n \in \mathbb{N}}$ and
study the model 
\begin{align} \label{e:mod2}
X_n = \varepsilon_n +  \sum_{k=1}^K \mu^{(k)} \cdot\,\mathbb{I}\{c_{k-1} < n \le c_k\}, \qquad n=1,...,N, 
\end{align}
where $N$  is the sample size, 
$(\varepsilon_n)_{n \in \mathbb{N}}$ is a centered error process, the change point locations are given by 
\begin{align} \label{e:mod1}
c_0 := 0<c_1<...<c_K< c_{K+1}:=N,
\end{align}
and
 $\mu^{(1)}, \ldots ,  \mu^{(K+1)} \in \mathbb{B}$ are $K+1$ 
  mean parameters  with $\mu^{(k)} \neq \mu^{(k+1)}$ for all $k$
 (the case $K=0$ corresponds to no changes). The time series of noise variables $(\varepsilon_n)_{n \in \mathbb{N}}$ is weakly dependent in an appropriate sense, as  e.g. quantified by decaying mixing coefficients or $m$-approximability. We do not, however, assume that the errors are stationary. All relevant parameters in this model may also depend on $N$ (such as the error distribution as well as $K$, $c_k$, $\mu^{(k)}$), but we usually do not make this dependency explicit in our notation. We also define the distance of the $k$th change point to its neighbors and the magnitude of that change as
\begin{align} \label{e:mod3}
\delta_k := \min\big(c_{k+1}-c_{k}, c_k-c_{k-1}\big)\qquad \textnormal{and}\qquad \Delta_k := \|\mu^{(k)}-\mu^{(k-1)}\|,
\end{align}
respectively. Our methodology is presented for general, separable Banach spaces. 
%\HD{{\bf würde ich weglassen, oder erst in Kap.2 erwähnen!} under a high-level condition on the convergence of the partial sum process of model errors (see Assumption \ref{ass_1} below), which can be verified in many different scenarios.} 
In particular, our theory includes multivariate data, as well as important spaces from Functional Data Analysis, such as the infinite-dimensional separable Hilbert space (typically expressed as the square integrable functions $L^2([0,1])$) and the space of continuous functions $\mathcal{C}([0,1])$, equipped with the supremum norm. Our main inferential aims are (in increasing order of difficulty):
\begin{itemize}
    \item[1)] To generate discrete intervals of integers $I_1, I_2,... \subset\{1,...,N\}$, where each interval includes at least one change, with some globally controlled error probability $\alpha$. 
    \item[2)] Under  general assumptions on the interplay of $\delta_k$ and $ \Delta_k$ to
guarantee that each interval contains exactly one change. 
\end{itemize} 
Evidently, smaller intervals for the localization of the changes are preferable and we will also establish a connection between lengths of the intervals, error structure and the pairs $(\delta_k, \Delta_k)$. 
\smallskip

\newpage 
\textbf{Key contributions} We  outline the most important innovations of this work.

\begin{itemize}
    \item[(i)] We introduce a novel  multi-scale  procedure for change point detection which is based  on a 
    \textit{Hölder-type scan statistic} defined by 
    \begin{align}
        \label{det1}
\gamma(n,h) := \frac{\|S_{n-h+1}^n-S_{n+1}^{n+h}\|}{N^{1/2} \rho(h/N)}, 
    \end{align}
where $  S_{i}^j := \sum_{\ell=i}^j X_\ell$ is the sum of the observations $X_i , \ldots , X_j$
and $\rho: [0,1] \to \mathbb{R}$ is an appropriate weight function 
(see eq. \eqref{polyweights} and \eqref{logweights} for some examples). 
We then use $\gamma(n,h) $ to scan for  potential change points (different $n$) in a multi-scale way (different $h$). Our method 
is  very flexible and can be used for data with strong or weak moment assumptions and for time series with weaker or stronger dependence. Moreover, it also  extends to non-stationary data.
\item[(ii)] The statistic  in \eqref{det1} is of Hölder-type in the sense that it is a continuous transform of the partial sum process w.r.t. the $\rho$-Hölder norm. The probabilistic core of our theory is thus leveraging convergence of the partial sum process in the space of  $\rho$-Hölder continuous functions. 
This approach is quite different from most existing methods that use other scalings than the Hölder-type statistic and that employ Gaussian approximations, only available for Euclidean data.
    \item[(iii)] Our theory is formulated for general, separable Banach spaces, including the most important ones of Functional Data Analysis (FDA). Related works in FDA typically assume that all changes are of the same magnitude. In contrast, our approach works under general assumptions on $K$,  $\delta_k$ and $\Delta_k$ (number, separation and magnitude of changes). If the model errors are Gaussian, our procedure works under the weakest possible assumptions on $(\delta_k,\Delta_k)$, which existed so far only for real-valued data. 
    \item[(iv)] Our theory is based on the high-level assumption that the partial sum process of errors converges in some Hölder space. We validate this high-level assumption for some concrete examples, including non-stationary time series in $\mathcal{C}([0,1])$, the space of continuous functions.
    \item[(v)] As an application of our theory, we study multiple change point detection in the distribution functions, for a panel of real-valued observations. This problem naturally allows an FDA formulation with non-stationary time series of model errors.
\end{itemize}

\textbf{Paper structure} The rest of this paper is organized as follows: The remaining section of the Introduction contains a discussion of related literature and some mathematical preliminaries for our subsequent discussion. In Section \ref{sec_2}, we present our main methodological and theoretical contributions. Section \ref{sec3} contains a mathematical application to panels of data points with changes in the distribution function. 
Section \ref{sec:sim} contains a simulation study and Section \ref{sec:data} an application of our new methodology to datasets in finance. Finally, the Appendix is dedicated to proofs and technical details.

\subsection{Related literature} 
The literature on change point analysis in general and multiple change point detection in particular (also called segmentation) is vast, and we restrict ourselves to the work which is most closely related to this paper. For helpful reviews on change point detection, we refer the reader to 
\cite{aue:horvath:2012},  \cite{jandhyala:2013},  \cite{Niuetal2016}, \cite{truongetal2020},  
\cite{cho:kirch:2024} 
 and  \cite{aue:kirch:2024}.  
 The breadth of the literature on multiple change point detection reflects the manifold applications of piecewise stationary models, such as in climate data, finance or neuroscience. 

\textbf{Segmentation for scalar and multivariate data} Borrowing terminology from \cite{cho:kirch:2024}, multiple change point problems can be addressed by either local or global methods. In local methods, small  subsamples of the data are investigated, using a method for the detection of a single change point. An example is the Binary Segmentation  (BS) method \citep{vostrikova:1981,venkatraman:1992}, where a CUSUM statistic is employed to sequentially break down the time series into smaller fragments, each containing only one change. In a seminal paper, \cite{fryzlewicz:2014} introduced a randomized version of BS, called Wild Binary Segmentation (WBS), which exhibits improved performance against closely spaced changes of small magnitude. In the notation of the present paper, it requires that $\min_k(\delta_k)\cdot \min_k(\Delta_k^2)\gg \log(T)$ for i.i.d. Gaussian errors, which is essentially optimal  \citep[see][for details]{arias-castro:candes:durand:2011,cho:kirch:2024}.  Variants of WBS have been considered in numerous works
\citep[see, for example][]{fryzlewicz:2020,cho:fryzlewicz:2024} and  BS and WBS-methods have also been investigated in the high-dimensional regime \citep[see, for example,][]{ChoFryzlewicz2015,wangyurinaldo2021,KocacsBuehlmannLiMUnk2024}. A related approach that focuses on the construction of confidence intervals for individual changes was proposed by \cite{fryzlewicz:2024}.
 A different, local approach for segmentation are moving sum (MOSUM) statistics, such as presented in \cite{eichinger:kirch:2018} or \cite{ChoKirch2021}. The problem formulation considered by these authors is related to our work, as it encompasses serially dependent data and does not require the model errors to be subgaussian. The MOSUM approach uses a scan statistic of the form \eqref{det1}, where the denominator is replaced by $h^{1/2}$. 

 In contrast to local procedures, global procedures are optimizing a 
 risk function subject to a penalization involving all possible segmentation at once. 
Different risk functions and different penalties then yield different segmentation methods. Typical examples are $\ell_0$-penalization \citep[see][for an early and a rather recent reference]{YAO1988,Romanoetal2022} and $\ell_1$-penalization \citep[see][among others]{Harchaouietal,Linetal2017}. Alternative approaches are provided by multiscale methods, minimizing the number of jumps of a piecewise constant fit of the mean function subject to a restriction defined by a multiscale statistic (see \citealp{frick:munk:sieling:2014} and \citealp{dette:eckle:vetter:2020} for an independent and the dependent case, respectively). The recent work of \cite{koehne:mies:2025} proposes a refined probabilistic analysis to obtain optimal detection rates with multiscale statistics for subgaussian and dependent scalar time series. The analysis also draws on the theory of Hölder spaces and is in this sense related to the present paper.
A detailed discussion of the various local and  global segmentation methods including an empirical comparison can be found  in the recent review paper of \cite{cho:kirch:2024}.
We point out that nearly all work on multiple change point detection focuses on (vectors of) scalar observations. There are few exceptions developing methods for functional data, which will be discussed next. 

\textbf{Change point detection in FDA} In FDA, observations are
random elements in some function space such as the square integrable functions or continuous functions on the unit interval. 
 The problem of detecting and locating changes in functional time series has been most intensively studied in models with at most one change. 
Tests for single mean changes were considered among others in 
\cite{berkes:gabrys:horvath:kokoszka:2009}, \cite{horvath:kokoszka:rice:2014},\cite{gromenko:kokoszka:reimherr:2017} and \cite{dettekokotvolgushev2020}. Structural breaks of time series in the space $\mathcal{C}([0,1])$
were studied in \cite{dette:kokot:aue}. Changes in the covariance  structure have been considered in \cite{stoehr:aston:kirch:2021},  \cite{horvath:rice:zhao:2022} and \cite{dettekokot2022}. 
Extensions of the traditional one change point model include heavy tailed data by \cite{mikosch:rackauskas:2010} and tests for epidemic changes (\cite{rackauskas:suquet:2006}, \cite{aston:kirch:2012}). The approach of \cite{rackauskas:suquet:2006} is related to this work, since it studies dyadic increments of the partial sum process, which can be understood as a Hölder-type statistic for epidemic changes. 

The FDA literature on multiple change point detection for a large number $K$ of changes  is fairly limited. The most relevant points of comparison seem to be \cite{chiou:chen:hsing:2019} and \cite{rice:zhang:2022}, where the second reference relies on Binary Segmentation. Both papers place quite restrictive assumptions on the model, requiring changes to be far apart from each other and hence also (implicitly) limiting the number of detectable changes. 
A method that is related but not directly comparable was proposed by \cite{bastian:basu:dette:2024} who deliberately search only for changes of sufficiently large magnitude. Since this paper uses Binary Segmentation and the theory of  \cite{rice:zhang:2022} as a subroutine, it imposes similar restrictions on the model.

\section{A multiscale method for change point detection} \label{sec_2}

\subsection{Notations and mathematical preliminaries} %\label{sec:12}

\textbf{$\mathbb{B}$-valued random variables} Throughout this work, $(\mathbb{B}, \| \cdot \|)$ denotes a real, separable Banach space of finite or infinite-dimension, with topological dual space $\mathbb{B}'$.   When formulating results specifically for Hilbert spaces, we use the notation $\mathbb{B}=\mathbb{H}$. For some probability space $(\Omega, \mathcal{A}, \mathbb{P})$, a random variable on $\mathbb{B}$ is a measurable map $X: (\Omega, \mathcal{A}, \mathbb{P}) \to   (\mathbb{B}, \| \cdot \| )$.  Notice that, since $\mathbb{B}$ is assumed to be separable, $X$ is automatically tight.
The expectation $ \mathbb{E}[X]=\mu \in \mathbb{B}$ is given by the Bochner integral and well-defined if $\mathbb{E}\|X\|<\infty$.
Similarly, we can define the covariance function of a random variable $X$ with arguments $ f,g \in \mathbb{B}'$  by 
\[
c^X[f,g] := \mathbb{E}\big[f(X)g(X)\big] -  \mathbb{E}\big[f(X)\big]\mathbb{E}\big[g(X)\big],
\]
which is well-defined if $\mathbb{E}\|X\|^2<\infty$. For details on Banach-valued random variables we refer to \cite{ledoux:talagrand:1991} and \cite{jankai2015}.

\textbf{$\mathbb{B}$-valued functions} Key technical tools to study the properties  of partial sum statistics are weak convergence results for sequential sum processes. To formulate such results we consider  the space of continuous $\mathbb{B}$-valued  functions
\[
\mathcal{C}([0,1], \mathbb{B}):=\{f: [0,1] \to \mathbb{B}| \,\, f \,\,\textnormal{continuous}\} , 
\]
which, equipped with the supremum-norm
\[
\|f\|_\infty:= \sup_{x \in [0,1]} \|f\|,
\]
is again a Banach space. We also consider   the closed subspace $\mathcal{C}^\rho([0,1], \mathbb{B})$ of ($\rho$-)Hölder continuous functions, where
$\rho : [0, 1] \to \mathbb{R}$  is a non-decreasing function   with $\rho(0)=0$, $\rho(x)>0$ for $x \in (0,1]$ that is right-continuous at $0$. More precisely, we define the ($\rho$-)Hölder norm as
\begin{align} \label{e:Hnorm}
\|f\|_{\rho} = \|f\|_\infty+ \sup_{0 \le x <y \le 1} \frac{\|f(y)-f(x)\|}{\rho(y-x)},   
\end{align}
and the corresponding Hölder space $\mathcal{C}^{\rho}([0,1], \mathbb{B})$ as the set  of all $\mathbb{B}$-valued functions $f \in \mathcal{C}([0,1], \mathbb{B})$  satisfying  both
\begin{align*}%\label{e:Hcond}
\|f\|_\rho < \infty, \qquad \textnormal{and} \qquad \lim_{ \delta \downarrow 0} \sup_{0\le x <y \le x+\delta \le 1} \frac{\|f(y)-f(x)\|}{\rho(y-x)} =0.
\end{align*}
To avoid trivial spaces (consisting only of constant functions) we will only consider choices of $\rho$ where for some constant $c>0$ and all $x \in [0,1]$ we have $\rho(x) \ge cx$.
Equipped with the Hölder norm $\|\cdot\|_\rho$, the space $\mathcal{C}^{\rho}([0,1], \mathbb{B})$ is also a separable Banach space. In particular, we can consider random variables in $\mathcal{C}^{\rho}([0,1], \mathbb{B})$, such as the $\mathbb{B}$-valued Brownian motion. This is a centered Gaussian process with independent, stationary increments  \citep{kuelbs:1973}. The fact that the $\mathbb{B}$-valued Brownian motion  is also an element of the more restrictive Hölder space $\mathcal{C}^{\rho}([0,1], \mathbb{B})$ for certain choices of weight functions (in particular those used in this paper) has been proved by  \cite{rackauskas:suquet:2009}.

\textbf{Further notations} 
In this work, we use a notion of asymptotic domination of a sequence of positive numbers. Namely, for $(a_n)_{n \in \mathbb{N}}$ and  $(b_n)_{n \in \mathbb{N}}$, we say $a_N \ll  b_N$ if $a_N/b_N \to 0$ as $N \to \infty$.

\subsection{Statistical methodology}
In this section, we develop inference methods for change point detection in the mean parameter  of  a Banach space-valued time series and prove its validity in the statistical model \eqref{e:mod2}. Recall that the locations of the changes are denoted by  $c_1 < \ldots < c_K$ (in other words  $c_k$ is  the location of the $k$-th change and  $K$ is  the unknown number of change points) and define for
  formal reasons $c_0:=0$, $ c_{K+1}:=N$.  Moreover, 
  $\mu^{(k)}$ is the mean parameter between $c_{k-1}$ and $c_k$, $\delta_k=\min\big(c_{k+1}-c_{k}, c_k-c_{k-1}\big)$ the separation of the $k$th change point and its neighbors and  $\Delta_k=\|\mu^{(k)}-\mu^{(k-1)}\|$ the magnitude of the $k$th change. Note that the  model parameters $K,\mu^{(k)}, c_k, \delta_{k}, \Delta_k$ may depend on $N$.
  
\textbf{Test statistics} In the following, we consider two types of weight function $\rho$ in the definition of the  weighted scan statistics in \eqref{det1}:
\begin{itemize}
    \item[i)]  Polynomial weights 
    \begin{align}
        \label{polyweights}
    \rho^{poly, \beta}(x)= x^{\beta}, \qquad \beta \in [0,1/2).
        \end{align}
    \item[ii)] 
Weights  with a logarithmic factor 
\begin{align}\label{logweights}  
    \rho^{log, \beta}(x)= x^{1/2} \log^\beta(x^{-1}), \qquad \beta \in (1/2, \infty),
  \end{align}
  which we call {\it logarithmic weights} throughout this paper.
\end{itemize} 
Dividing by these weights puts more emphasis on small values of $x$.
In this regard, logarithmic weights are stronger than polynomial weights and in both classes, a value of $\beta$ closer to $1/2$ corresponds to stronger weighting.
Next, we recall  the definition of the partial sums $S_{i}^j := \sum_{\ell=i}^j X_\ell$ and scan statistic 
\begin{align}\label{e:gam}
 \gamma(n,h) := \frac{\|S_{n-h+1}^n-S_{n+1}^{n+h}\|}{N^{1/2} \rho(h/N)}.
\end{align}
Notice that in the simple case where we choose polynomial weights with $\beta=0$, we obtain a standardization by $1/\sqrt{N}$. Stronger weighting blows up $\gamma(n,h)$ for small values of $h$, which  corresponds to relying more strongly on local information. Heuristically, such a weighted procedure will have an easier time spotting closely spaced changes, where small values for  $h$ are used, but it will also be less robust to outliers that impact local sums $S_{n+1}^{n+h}$ more strongly.
Now, using $\gamma(n,h)$, we can define a statistical procedure to locate changes in the data sample.

\textbf{A detection algorithm} We use the test statistic $\gamma(n,h)$ to scan for changes at different locations $n$ and for different widths $h$. Therefore, we define the set of all permissible tuples 
\[
\mathbf{N}^{all}:=\{(n,h): 1 \le n-h+1<n+h\le N\}.
\]
It is appropriate to say that
\begin{align*}
     \textnormal{an element}\quad  (n, h) \in \mathbf{N}^{all} \quad \textnormal{"characterizes the  interval} \quad[n-h+1,n+h]\textnormal{"}
\end{align*}
in the sense that $\gamma(n,h)$ draws on data indexed in this interval. In this way we will sometimes identify pairs $(n,h)$ with intervals. 
On the set $\mathbf{N}^{all}$ we define an ordering 
\begin{align} \label{e:def:order}
(n,h)\preceq (n',h')  \quad :\Leftrightarrow \quad \begin{cases}
    h < h',\\
    h = h' \,\, \textnormal{and} \,\,\,\,n<n'.
\end{cases}
\end{align}
Looking at this ordering for the statistic $\gamma(n,h)$, this means that shorter sums (small $h$) will take precedence over longer ones and, somewhat arbitrarily, centers to the left (smaller $n$) take precedence over centers to the right. In Algorithm \ref{alg1} below, we will run a for-loop over $(\mathbf{N}, \preceq )$, for a subset $\mathbf{N} \subset \mathbf{N}^{all}$. By this we mean that we go through all pairs $(n,h) \in \mathbf{N}$ from smallest  to largest pair (according to $\preceq $). 
Next, we define for a set $\mathbf{N} \subset \mathbf{N}^{all}$ and for a pair $(n, h)$  the subset 
\begin{align}\label{e:def:minus}
    \mathbf{N} \circleddash (n,h) := &\{(n',h') \in \mathbf{N}: (n,h) \preceq (n',h'),\\
    & [n'-h'+1, n'+h'] \cap [n-h+1, n+h] = \emptyset\}. \nonumber 
\end{align}

The operation "$\circleddash$" thus removes all tuples that are smaller  than $(n,h)$ according to $\preceq$ and those whose defined interval overlaps with $(n,h)$.
We need one more piece of notation and define 
\begin{align}
    \label{hol1}
n^{max}_{(n,h)}:= \textnormal{argmax}\{\gamma(n',h): n-h+1<n'<n+h, (n',h) \in \mathbf{N}\}.
\end{align}
Below, we define Algorithm \ref{alg1}, that scans an index set $\mathbf{N}$ for changes. The other input parameters of the algorithm  are a critical threshold $q$ and a set of suspected changes $\mathbf{C}$, which is updated continuously and  initialized as $\mathbf{C}=\emptyset$. Notice that implicitly, the algorithm also depends on the definition of the scan statistic $\gamma$ and thus on the weight function $\rho$. In the algorithm, roughly speaking, the scan $\gamma(n,h)$ 
 proceeds along the ordering  $\preceq$. If a value of $\gamma(n,h)$ surpasses $q$, the algorithm searches the neighborhood around $n$ for the maximum value of $\gamma(n',h)$, adding the maximizing value $(n^{\max}_{(n,h)},h)$ to the set of suspected changes $\mathbf{C}$. Then, the set $\mathbf{N} $ is updated to $\mathbf{N} \circleddash (n^{\max}_{(n,h)},h)$. This means that we will not search for changes in intervals intersecting with $[n^{\max}_{(n,h)}-h+1, n^{\max}_{(n,h)}+h]$ in the future, because we have already found a change point candidate there. The algorithm then calls itself to continue the  search in the smaller set $\mathbf{N} \circleddash (n^{\max}_{(n,h)},h)$, until the set is empty. 

\begin{algorithm}
\caption{MultiScan} \label{alg1}
\begin{algorithmic}[1]
    \State \textbf{Input:} $\mathbf{C} \subset \mathbf{N}^{all}$, $\mathbf{N} \subset \mathbf{N}^{all}$, $q>0$
    \State \textbf{Output:} $\mathbf{C}$ (set characterizing intervals with changes)
    
    \If{$\mathbf{N} = \emptyset$}
        \State \textbf{Return} $\mathbf{C}$
        \State \textbf{Terminate Algorithm}
    \EndIf

    \For{$(n, h) \in \mathbf{N}$ along $\preceq$}
        \State Calculate $\gamma(n, h)$ in \eqref{e:gam}
        \If{$\gamma(n, h) > q$}
            \State Calculate $n^{\max}_{(n, h)}$ in \eqref{hol1}
            \State Update $\mathbf{C} \gets \mathbf{C} \cup \{(n^{\max}_{(n, h)},h)\}$
            \State Update $\mathbf{N} \gets \mathbf{N} \circleddash (n^{\max}_{(n,h)},h)$
            \State \textbf{Break} \Comment{Exits the entire loop}
        \EndIf
    \EndFor
    
    \State \textbf{Recursively call} MultiScan($\mathbf{C}$,$\mathbf{N}$,$q$) \Comment{Pass updated parameters}
    
\end{algorithmic}
\end{algorithm}

\noindent \textbf{Inference aims} Before we proceed to the theoretical analysis, we will state two  statistical aims for Algorithm \ref{alg1},
one weaker than the other one. Let $\alpha \in (0,1)$ be a user-determined bound for the global false detection rate and $q=q(\alpha)$ a suitably chosen threshold. Then it should hold for the outputs of MultiScan$(\emptyset,\mathbf{N},q)$ that:
\begin{itemize} 
    \item[(i)] \textit{Weak localization:} 
    \begin{align}
        \label{e:weakl}
        \liminf_{N \to \infty}  \mathbb{P}\big(~\forall (n,h) \in \mathbf{C} \,\,\exists k \in \{1,...,K\}: n-h+1 \le  c_k \le n+h~\big)\ge 1-\alpha.
    \end{align}
    \item[(ii)] \textit{Strong localization:} 
Condition \eqref{e:weakl} holds. Moreover, with  probability converging to $1$, there exists (simultaneously for all $k$) for each change $c_k$ a unique pair of output parameters $(n,h)=(n_k,h_k) \in \mathbf{C}$ such that
\begin{align} \label{e:strongl}
    n_k-h_k+1\le c_k \le n_k+h_k.
\end{align}
\end{itemize}
Weak localization means that,  with high probability  ($\ge 1-\alpha$), any of the intervals $[n-h+1, n+h]$ from the output of Algorithm \ref{alg1}
contains at least one change point.
Strong localization implies additionally that every change point $c_k$ is captured by one (unique) interval $[n-h+1, n+h]$, with probability going to $1$. Strong localization implies the more typical statistical guarantee from the literature \cite[see, for example,][]{fryzlewicz:2024,koehne:mies:2025} that
\begin{align} \label{e:frc}
\liminf_{N \to \infty}  \mathbb{P}\big(~\forall k \in \{1,...,K\} ~\exists (n,h) \in \mathbf{C} : n-h+1 \le  c_k \le n+h, |\mathbf{C}|=K\big)\ge 1-\alpha.  
\end{align}
This means that all changes are captured by distinct intervals with (asymptotic) probability $\ge 1-\alpha$. Strong localization additionally explains what happens in the remaining $\le \alpha$ proportion of cases: Here it may be the case that $|\mathbf{C}|>K$ (more intervals than changes). It is still guaranteed that all $c_k$ are captured by unique intervals -  however, there now can be some intervals that include no changes. In this sense, strong localization guarantees that asymptotically no  change point is overlooked. To obtain strong localization, rather than just \eqref{e:frc}, the order of testing multiscale hypotheses (in Algorithm \ref{alg1}) plays an important role. \\
Obviously, strong localization can only hold if the changes are not too small and sufficiently separated. These requirements are not necessary for weak localization. 

\subsection{Theoretical analysis}

\textbf{Hölder space convergence} We now show that Algorithm \ref{alg1} satisfies the weak and strong localization property.
For this purpose, we first define the (continuous) sequential sum process of model errors
\begin{align} \label{e:lin}
P_{N}(x):= N^{-1/2} \sum_{i=1}^{\lfloor x N \rfloor} \varepsilon_{i}+N^{-1/2} \big(Nx-\lfloor N x \rfloor\big) \varepsilon_{\lfloor N x \rfloor+1}, \qquad x \in [0,1].
\end{align}
In the following, we will always tacitly assume that the errors $(\varepsilon_n)_{n \in \mathbb{N}}= (\varepsilon_{n,N})_{n \in \mathbb{N}}$ form a triangular array. Recall that the other model parameters $K,\mu^{(k)}, c_k, \delta_{k}, \Delta_k$ also may depend on $N$  and we usually suppress this dependence in favor of a lighter notation.

\begin{assumption}\label{ass_1}\textnormal{
    For the weight function $\rho$ we have weak convergence
    \begin{align} \label{e:holder}
    \{P_{N}(x): x \in [0,1]\}\overset{d}{\to}\{W(x): x \in [0,1]\}
    \end{align}
    in the Hölder space $\mathcal{C}^\rho([0,1], \mathbb{B})$. Here, $W$ is a centered Gaussian process, characterized by the continuous (in all arguments) limiting (long-run) covariance function
    \begin{align}
        C^{\varepsilon}(x,y)[f,g]:= \lim_{N \to \infty}\mathbb{E}\big[f[P_{N}(x)] g[P_{N}(y)]\big], \qquad x,y \in [0,1], \quad f, g \in \mathbb{B}'.
 \label{det11}
    \end{align}
    }
\end{assumption}

\noindent Notice that 
-  in contrast to most of the literature - Assumption \ref{ass_1} (and as consequence our theory) does not require the error process  to be stationary. In fact, to the best of our knowledge,  only two references in this direction exist, namely \cite{pein:sieling:munk:2017} and \cite{koehne:mies:2025}, both of which focus on univariate time series. The former consider independent Gaussian data  with a time varying variance, while the latter allow weak time dependence for subgaussian data. 
Our theory presented below is developed for general moment and dependence assumptions and covers non-stationary  time series in function spaces (see Theorem \ref{propmixing}). A detailed example of non-stationary data is discussed in Section \ref{sec3}, where we also validate Assumption  \ref{ass_1}.

In the case where the noise variables $(\varepsilon_n)_{n \in \mathbb{N}}$ form a stationary time series, the weak convergence in Assumption \ref{ass_1} is called a \textit{Hölderian invariance principle} and the limiting distribution $W$ is a Brownian motion on the Banach space $\mathbb{B}$. Hölderian invariance principles have been studied in the literature, first for independent, real valued data (\cite{lamperti:1962}) and subsequently for dependent time series by (\cite{hamadouche:2000} and \cite{giraudo:2017}). An interesting feature of such results is that they can be shown for general dependence concepts and on fairly general spaces. We give an example relevant for FDA, where often data are modeled as elements of the separable Hilbert space of square integrable functions $\mathbb{H}=L^2([0,1])$. A more comprehensive overview of convergence on Hölder spaces is given in Section \ref{sec23}

\begin{example}\label{ex1}
\textnormal{ Consider a sequence of i.i.d. random variables $(\psi_n)_{n \in \mathbb{Z}}$ in a separable Hilbert space $\mathbb{H}$ and a sequence of continuous linear operators $(A_i)_{i \in \mathbb{Z}}$ acting on $\mathbb{H}$. We impose that $\sum_i A_i\neq 0$ and that $\sum_i \|A_i\|_\mathcal{L}<\infty$, where $\|\cdot\|_\mathcal{L}$ denotes the operator norm. We then define the error process $\varepsilon_n := \sum_{i} A_i \psi_{n-i}$. Furthermore, suppose that one of the following conditions holds
\begin{itemize}
    \item[i)] Suppose that $\rho(x)=x^\beta$ is the polynomial weight function and that
    \[
    \lim_{t \to \infty} t^{(1/2-\beta)^{-1}}\mathbb{P}(\|\psi_1\|>t)=0.
    \]
    \item[ii)] Suppose that $\rho(x)=x^{1/2}\log^\beta(x^{-1})$ is the logarithmic weight function and that
    \[
    \mathbb{E}\exp\big(t\|\psi\|^{1/\beta}\big)<\infty, \qquad \forall t>0.
    \]
\end{itemize}
    Then, according to \cite{rackauskas:suquet:2009},  the weak convergence \eqref{e:holder} holds in the space 
     $\mathcal{C}^\rho([0,1], \mathbb{H})$
 and the limit $W$ is an $\mathbb{H}$-valued Brownian motion. We also notice that the respective decay assumptions in the polynomial and logarithmic case (i) and ii) respectively) are sharp, implying that stronger weighting (say logarithmic instead of polynomial) requires stronger moment assumptions.}
\end{example}

As noted above, Assumption \ref{ass_1} can be shown for  non-stationary time series of errors. However, in practice, the user has to estimate the (long-run) 
covariance function \eqref{det11} of the errors. In the independent case, this is possible using first order differences of the data. In the case of dependent but stationary data,  one can use blocking techniques, as proposed e.g. in \cite{wu:zhao:2007}, Section 5 (see also \cite{dette:eckle:vetter:2020}, Proposition 1). For non-stationary and dependent noise, estimation is also possible, but requires larger data samples and is therefore often less practical than either proposing a model with stationary or independent errors.
 
\noindent \textbf{Limiting distribution and quantiles} 
The convergence behavior of the MultiScan algorithm depends on the input set $\mathbf{N} \subset \mathbf{N}^{all}$. For detection purposes taking the entire set $\mathbf{N}^{all}$ as input is ideal, as larger subsets are associated with  more powerful detection, but this also entails a computationally expensive procedure with $\mathcal{O}(N^2)$ operations. It can thus be more practicable to use a thinned out version of $\mathbf{N}^{all}$, and an obvious choice is the following pyramid structure for some $\theta>1$
\begin{align*}
    \mathbf{N}_N^\theta= \{(n,h) \in \mathbf{N}^{all}: \exists m \in \mathbb{N} \,\, s.t.\,\, h=\lfloor \theta^m\rfloor \}.
\end{align*}
Evidently $|\mathbf{N}_N^\theta| = \mathcal{O}(N \log(N))$ yielding a computationally less costly procedure, which as we will see below maintains high accuracy. To describe the limiting distribution of our statistic, it is useful to consider the 
symmetric triangle with  the base on the unit interval   
\begin{align} \label{e:def:tri}
\blacktriangle_{0,1}:=  \big\{(x,y): x \in [0,1], y \in (0,\min\{x, (1-x)\}]\big\}.
\end{align}
Therewith we define the 
random variable 
\begin{align} \label{e:defLrho}
L^\rho:= \sup_{(x,y) \in \blacktriangle_{0,1}} \frac{\|W(x+y)-2W(x)+W(x-y)\|}{\rho(y)},
\end{align}
where $W$ is the centered Gaussian process from Assumption \ref{ass_1}, and we denote  
the upper $\alpha$-quantile of its distribution  by
$q_{1-\alpha}$. 

\textbf{Weak localization} The following result shows that weak localization 
in the sense of \eqref{e:weakl}
is obtained by Algorithm \ref{alg1} if Assumption \ref{ass_1} holds. Throughout  this section, we will assume that the index set $\mathbf{N}_N$ is for all $N$ of the form 
\begin{align} \label{e:setchoice}
\qquad \textnormal{either} \qquad\mathbf{N}_N =  \mathbf{N}^{all}, \qquad \textnormal{or} \qquad \mathbf{N}_N =  \mathbf{N}_N^\theta .
\end{align}
However, from the proofs in the appendix  it  will be  clear that for other choices of index sets related results hold.

\begin{theo} \label{theo1}
 {\it    Suppose that Assumption \ref{ass_1} is satisfied  and let  $\mathbf{N}_N $ be either  $\mathbf{N}^{all}$ or $  \mathbf{N}_N^\theta$. Then, for any $\alpha \in (0,1)$, the output of MultiScan$(\emptyset, \mathbf{N}_N, q_{1-\alpha})$ satisfies the weak localization property \eqref{e:weakl}. Furthermore, if $K=0$, then \eqref{e:weakl} 
    holds with 
 "$\liminf_{N \to \infty} $"  replaced by the ordinary "$\lim_{N \to \infty}$" and the inequality "$\ge 1-\alpha$"  replaced by the equality "$=1-\alpha$".}
\end{theo}

\textbf{Strong localization} In order to verify strong localization properties we need to impose additional assumptions on the separation $\delta_k$ and magnitude $\Delta_k$ of the changes, which are defined in \eqref{e:mod3}.
\begin{assumption} The mean parameters $\mu^{(k)} \in \mathbb{B}$ are uniformly (with respect to $k$ and $N$) bounded. Moreover, one of the following conditions holds. \label{ass_2}
    \begin{itemize}
        \item[i)] If $\rho$ is the polynomial weight function \eqref{polyweights}, then  it holds that
        \begin{align} \label{e:ass_2_pol}
        \min_{k=1,...,K} \delta_k^{1-\beta}\Delta_k\gg N^{1/2-\beta}.
         \end{align}
        \item[ii)] If  $\rho$ is the logarithmic weight function \eqref{logweights}, then  it holds that
         \begin{align}\label{e:ass_2_log}
        \min_{k=1,...,K}\delta_k^{1/2}\Delta_k\gg \log^\beta(N).
         \end{align}
        
    \end{itemize}
\end{assumption}

Assumption \ref{ass_2} implies an interplay between the weight function $\rho$ and the conditions on $\delta_k$ and $ \Delta_k$. 
 We can see that stronger weighting (say logarithmic instead of polynomial weights) allows $\delta_k^{1/2}\Delta_k$ to be smaller. Importantly, this implies a trade-off between the tails of the noise and the signal strength, because stronger weighting is only possible if the noise satisfies the $\rho$-Hölder space convergence in Assumption \ref{ass_1}, which in turn requires higher moments for stronger weighting (see the discussion in Example \ref{ex1}). We also notice that in the case of logarithmic weights with $\beta$ close to $1/2$ the requirement in \eqref{e:ass_2_log} is close to the minimax condition  for i.i.d. one-dimensional Gaussian data that we discuss further below.

\begin{theo}\label{theo2}
{\it  Suppose that Assumptions \ref{ass_1} and \ref{ass_2} are satisfied and let  $\mathbf{N}_N $ be either  $\mathbf{N}^{all}$ or $  \mathbf{N}_N^\theta$. Then, the output of MultiScan$(\emptyset, \mathbf{N}_N, q_{1-\alpha})$ satisfies the strong localization property \eqref{e:strongl}. More precisely, with  probability converging to $1$, there exists (simultaneously) for each change $c_k$ a unique pair of output parameters $(n,h)=(n_k,h_k) \in \mathbf{C}$ such that
\[
n_k-h_k+1\le c_k \le n_k+h_k
\]
where

\[
h_k \le C_0 \cdot 
\begin{cases}
    \bigg(\frac{N^{1/2-\beta}}{\Delta_k}\bigg)^{\frac{1}{1-\beta}}, \,\,\, \textnormal{for}\,\, \rho(x)=x^\beta \\
    \frac{\log^{2\beta}(N)}{\Delta_k^2}, \,\,\, \,\, \,\, \, \qquad \textnormal{for}\,\, \rho(x)=x^{1/2}\log^\beta(x^{-1}). \\
\end{cases}
\]
Here $C_0>0$ is some fixed constant depending on the covariance structure of the process  $W$ in Assumption \ref{ass_1}.
}
\end{theo}

   Theorem \ref{theo2} marks a significant progress for multiple change point detection, especially for functional data. Even for i.i.d. noise, no comparable  result exists that implies a general trade-off between the spacings $\delta_k$ and the magnitudes of the changes  $\Delta_k$. Consider, in contrast, the existing guarantees afforded by  \cite{chiou:chen:hsing:2019} and \cite{rice:zhang:2022}: The former assumes that $\min_k \delta_k \ge \epsilon N$ for some fixed $\epsilon>0$, while in the latter requires the slightly weaker condition of  $\min_k \delta_k\gg  N^{7/8}$. Neither method can give any formal guarantees for the detection of genuinely narrowly spaced changes. Their assumptions also strictly limit the total number of detectable changes in the data sample. In contrast to these works,  Theorem \ref{theo2} allows for narrow spacings that are growing only logarithmically in $N$.

\textbf{Optimality for Gaussian data} 
Theorems \ref{theo1} and \ref{theo2} hold for general  Banach space-valued time series. It is of interest to relate these to  the problem of optimal rates, which are generally derived for i.i.d. Gaussian noise. As pointed out by \cite{cho:kirch:2024}, there do not seem to be comparable optimality results for data with heavier tails or for dependent data. The conditions necessary for change point inference are called by these authors "separation rates" and are distinguished from "localization rates" needed for estimation. 
In the one-dimensional Gaussian case, the minimax separation condition is known to be as $K=K_N \to \infty$
\begin{align*} 
\min_{k=1,...,K}\delta_k^{1/2}\Delta_k\gg \log^{1/2}(K_N).
\end{align*}
The right side can be replaced by  $\log^{1/2}(N)$ if $K_N \sim N^{r}$ for some $r >0$. 
We can show that MultiScan$(\emptyset, \mathbf{N}_N, q_{1-\alpha})$
achieves the resulting optimal results (usually derived under the assumption of univariate Gaussian data) even for  Banach space-valued  time series. For this purpose, we use the logarithmic weight function with $\beta=1/2$, i.e. 
\[
\rho(x) = \rho^{\log, 1/2}(x) = x^{1/2}\log^{1/2}(x^{-1}),
\]
for the remainder of this section 
and assume the noise to be increments of some Gaussian process. More precisely, we impose:
\begin{assumption} There exists a Banach-valued Gaussian process $W$ such that:
\label{ass_5}
    \begin{itemize}
        \item[i)] The model noise satisfies
            \begin{align*}
   % \label{det5}
    \varepsilon_n := \sqrt{N}[W(n/N)-W((n-1)/N)].
     \end{align*}
     \item[ii)] There exists a constant $C_W<\infty$ such that almost surely
        \begin{align} \label{e:WCadd}
        \limsup_{ \delta \downarrow 0} \sup_{0\le x <y \le x+\delta \le 1} \frac{\|W(y)-W(x)\|}{\rho(y-x)}\le C_W. 
        \end{align}
    \end{itemize}
\end{assumption}

The case of i.i.d. Gaussian noise is covered when the process is a Banach-valued Brownian motion, but the statement remains valid for more general error processes. 
For the Banach-valued Brownian motion, the result follows by a generalization of Lévy's modulus of continuity theorem  \citep[see][]{acosta:1985}. It is important to remark that while the Brownian motion satisfies Assumption \ref{ass_5}, part ii), it does not live on the Hölder space $\mathcal{C}^{\rho}([0,1], \mathbb{B})$ (in which case $C_W=0$ had to be true) and hence different proof techniques need to be used. 
For the rest of this section, we will let $\alpha \in (0,1)$ be a nominal level and as before $q_{1-\alpha}$ be the upper $\alpha$-quantile of $L^\rho$. We additionally assume that $q_{1-\alpha}>C_W$ to exclude some pathological cases. 

\begin{prop} \label{prop1}
 {\it    Suppose that  
    \begin{align*} %\label{e:minimax} 
\min_{k=1,...,K}\delta_k^{1/2}\Delta_k\gg \log^{1/2}(N)
\end{align*}
    and that Assumption \ref{ass_5} holds. Then, it follows that:
    \begin{itemize}
        \item[i)] The output of MultiScan$(\emptyset, \mathbf{N}_N, q_{1-\alpha})$ satisfies the weak localization property \eqref{e:weakl}, and exact equality in \eqref{e:weakl} holds if no change occurs.
        \item[ii)] The strong localization property %\eqref{e:strongl}
        holds.
         More precisely,  with
probability converging to $1$, there exists (simultaneously)  for each change point location $c_k$  
         a unique pair of output parameters $(n,h)=(n_k,h_k) $  in the output of MultiScan$(\emptyset, \mathbf{N}_N, q_{1-\alpha})$ such that
\begin{align*}
%\label{det2}
n_k-h_k+1\le c_k \le n_k+h_k\quad \textnormal{and it holds that}\quad h_k \le C_0 \cdot 
\frac{\log^{}(N)}{\Delta_k^2}.%, \,\,\, \, \textnormal{for}\,\, \rho(x)=x^{1/2}\log(x^{-1}). 
\end{align*}
Here $C_0>0$ is a  fixed constant depending on the covariance structure of the Gaussian process  $W$.
    \end{itemize}
 }   
\end{prop}

Proposition \ref{prop1} shows that changes are detected under optimal conditions if the number of changes $K=K_N$ diverges. A similar but distinct condition exists in the case where the number of changes is small. More precisely, Proposition 2.1 in \cite{cho:kirch:2024}  
considers the problem of detecting an epidemic change in a univariate time series on a small segment. Here 
\begin{align} \label{e:formX}
\mathbb{E}X_i = \mu +\mu' \mathbb{I}\{c_1 \le i \le c_2\},
\end{align}
where $1 \le c_1<c_2 < N$, the length of the segment is $\delta= c_2-c_1+1$ and the magnitude of the change $\Delta = \|\mu'\|$. If $\delta/N \to 0$, it turns out that the two hypotheses
\[
H_0: \Delta=0 \qquad vs. \qquad H_1: \Delta>0
\]
are indistinguishable, if 
$\delta^{1/2}\Delta \le \log^{1/2}(N/\delta)-C_N$ for some sequence $C_N \uparrow \infty$. Similarly as before, we will confine ourselves to the most relevant special case of a small segment and thus assume that for some $r \in (0,1)$ it holds that $N/\delta \sim N^{r}$, which implies $\log^{1/2}(N/\delta) \sim \log^{1/2}(N)$.
We then define the condition for reliable distinguishability as 
\begin{align} \label{e:dist}
\delta^{1/2}\Delta \gg  \log^{1/2}(N),
\end{align}
and define the test decision for the hypotheses pair $H_0-H_1$, to reject $H_0$ if 
\begin{align*} %\label{e:decision}
    |\textnormal{MultiScan}(\emptyset, \mathbf{N}_N, q_{1-\alpha})|>0.
\end{align*}

\begin{prop} \label{prop2}
   Suppose the epidemic change point model \eqref{e:formX} and that
 Assumption \ref{ass_5} holds. Then, it follows under $H_0$ that 
 \begin{align*}
    \lim_{N \to \infty} \mathbb{P}_{H_0}(|\textnormal{MultiScan}(\emptyset, \mathbf{N}_N, q_{1-\alpha})|>0) = \alpha
 \end{align*}
 and under the alternative with \eqref{e:dist}  that
 \[
     \lim_{N \to \infty} \mathbb{P}_{H_1}(|\textnormal{MultiScan}(\emptyset, \mathbf{N}_N, q_{1-\alpha})|>0) = 1.
 \]
\end{prop}

\subsection{Convergence in the Hölder space} \label{sec23}

Assumption \ref{ass_1} requires the convergence of the partial sum process of errors in the Hölder space $\mathcal{C}^\rho([0,1], \mathbb{B})$. In this section, we discuss some known results when this condition holds.

\textbf{Literature for weak convergence in $\mathcal{C}^\rho([0,1])$} First, we consider extensions of Example \ref{ex1} to separable Banach spaces. More precisely, consider a sequence of linear operators $(A_i)_{i \in \mathbb{Z}}$ acting on $\mathbb{B}$ and satisfying the summability condition $\sum_i \|A_i\|_\mathcal{L}<\infty$. Moreover, denote by $(\psi_i)_{i \in \mathbb{Z}}$ a sequence of centered i.i.d. variables in $\mathbb{B}$.
Theorem 6 in \cite{rackauskas:suquet:2010} 
maintains that the linear process
$\varepsilon_n := \sum_{i} A_i \psi_{n-i}$ satisfies Assumption \ref{ass_1}, with the limit $W$ being a Banach space-valued Brownian motion, under the following two conditions:
\begin{itemize}
    \item[i)] The sequence $(\psi_i)_{i \in \mathbb{Z}}$ satisfies the Central Limit Theorem (CLT).
    \item[ii)] It holds  for every $\eta>0$ that $
    \lim_{t \to \infty} t\mathbb{P}(\|\psi_1\|>\eta t^{1/2}\rho(t^{-1}))=0$.
\end{itemize}
Condition ii) is essentially a moment condition and turns out to be sharp in the i.i.d. case where $\varepsilon_n=\psi_n$ \citep[see][]{rackauskas:suquet:2009}. In particular, this condition  implies that $\mathbb{E}\|\psi_1\|^2<\infty$. Condition i) implies that the errors $(\psi_i)_{i \in \mathbb{Z}}$ satisfy a CLT, which is then inherited by the linear process. In a classical work \cite{hoffmann-jorgensen:pisier:1976} showed that any sequence of i.i.d. square integrable random variables (and in particular the above sequence $(\psi_i)_i$) satisfies the CLT in a Banach space of type 2 - indeed this is an equivalence. An important example of such spaces are the $L^p([0,1])$ spaces for $p\ge 2$, where $p=2$ corresponds to Example \ref{ex1}. 
Convergence on $L^p$-spaces with $1 \le p \le 2$ has been considered, for example, in the classical note of \cite{zinn:1977}. Another important example of a Banach space that is not of type 2, but plays an important role in FDA, is the set of real valued continuous functions
$\mathcal{C}([0,1], \mathbb{R})$.  It is also of interest since any other separable Banach space is isometrically isomorphic to some (closed) subspace of $\mathcal{C}([0,1], \mathbb{R})$. \cite{jain:marcus:1975} demonstrated that if the stronger assumption $\mathbb{E}\|\psi_1\|_{\rho^{poly,r}}^2<\infty$ holds for some $r>1/2$, then $(\psi_i)_i$ already satisfies the CLT in $\mathcal{C}([0,1], \mathbb{R})$. 
More general results on the CLT in  $\mathcal{C}([0,1], \mathbb{R})$ follow from general empirical process theory, such as Theorem 2.2.4 in \cite{vaart:wellner:1996}.

 In general Banach spaces, there exist fewer results beyond independent variables and linear processes. 
 A recent work of \cite{kutta:doernemann:2025} considers Hölderian invariance principles in separable Hilbert spaces under the condition of  $L^p$-$m$-approximability   \citep[see][for a definition]{hormann:kokoszka:2010}. For real-valued data, various dependence concepts have been considered, including strong mixing \citep{hamadouche:2000}, \cite{giraudo:2017}, Maxwell and Woodroofe conditions \citep{giraudo:2018} and martingale difference schemes \citep{giraudo:2016}.

\textbf{Mixing time series of continuous functions} 
We contribute to the literature on weak  convergence in Hölder spaces, by establishing a result of the form \eqref{e:holder}
for the case $(\mathbb{B}, \| \cdot \|) =  \big ( \mathcal{C}([0,1]), \| \cdot \|_\infty \big )$, i.e. the space of continuous functions equipped with the supremum norm.  In contrast to all previously cited results, we do not require the errors to be stationary and hence obtain a limiting distribution that may not be a Brownian motion. Furthermore, it seems that previously no Hölder-type convergence has been derived for mixing functional time series, even in the stationary case. For a definition and discussion of mixing coefficients, we refer the reader to \cite{bradley:2005}.\\
Notice that in the space $\mathcal{C}([0,1], \mathbb{R} )$, each of the errors is itself a function, that is  $\varepsilon_n = \{\varepsilon_n(s): s \in [0,1]\}$ and the partial sum process can be interpreted as a continuous functions in two arguments $P_N(x,s):=P_N(x)(s) $. 
We impose the following assumptions. 

\begin{assumption} Let $q$ be an even integer and $p>q$ a number.
\label{assmixing}
    \begin{itemize}
        \item[i)] The centered errors $\varepsilon_n$ satisfy for some $r \in (0,1/2)$ with $q>1+1/r$ the moment conditions
        \[
        \sup_n\mathbb{E}\big[\|\varepsilon_n\|^p\big]<\infty, \qquad \sup_n\mathbb{E} \bigg[\bigg(\sup_{0<s<t<1}\frac{|\varepsilon_n(s)-\varepsilon_n(t)|}{|s-t|^r}\bigg)^p\bigg]<\infty.
        \]
        \item[ii)] The time series  $(\varepsilon_n)_{n \in \mathbb{N}}$ is $\alpha$-mixing, with mixing coefficients satisfying $\alpha(n) \le a b^{-n}$, for some constants $a>0$ and $b \in (0,1)$.
        \item[iii)] There exists a continuous covariance function $c^\varepsilon:[0,1]^4\to \mathbb{R}$ such that
        \[
        c^{\varepsilon}(x,y, s,t):= \lim_{N \to \infty}\mathbb{E}\big[P_{N}(x,s)P_{N}(y,t)]\big], \qquad x,y,s,t \in [0,1].
        \]
    \end{itemize}
\end{assumption}

\begin{theo} \label{propmixing}
{\it   Let Assumption \ref{assmixing} be satisfied and assume that $\rho $
   is the polynomial weight function with parameter $\beta$ such that  $q>(1/2-\beta)^{-1}$.
   Then it holds that 
    \begin{align*}
    \{P_{N}(x): x \in [0,1]\}\overset{d}{\to}\{W(x): x \in [0,1]\}
    \end{align*}
    in the Hölder space $\mathcal{C}^\rho([0,1],  \mathcal{C}([0,1]))$, where the limiting process $W$ is characterized by the covariance function $c^{\varepsilon}$.
    }
 \end{theo}

\section{An application to changes in distribution} \label{sec3}

We consider the problem of finding changes in the distribution functions of a data panel. This problem has been studied for time series by \cite{carlstein:1988} and more recently for entire panels  by \cite{horvath:kokoszka:wang:2021} in the context of sequential change point detection. We consider a similar setup as in \cite{horvath:kokoszka:wang:2021}, but in a retrospective formulation and with multiple changes.
More precisely, let $F_1, F_2,....,F_N$ be continuous cumulative distribution functions and suppose that for each $n=1,...,N$ we obtain observations

\[
Y_{n,1},...,Y_{n,M} \sim F_n 
\]
with all observations $Y_{n,m}$ independent across $n$ and $m$. We assume that
\[
F_n = F_N^{(j)}, \qquad \textnormal{for} \,\,\, c_{j-1}<n\le c_j,
\]
where as before $c_j$ denotes the change point location (see eq. \eqref{e:mod1}). For each $n$, we define the empirical distribution function
\[
\widehat{F}_{n,M}(t):= \frac{1}{M}\sum_{m=1}^M \mathbb{I}\{Y_{n,m}\le t\}\qquad \textnormal{and its scaled version} \qquad X_{n,M}:= \sqrt{M}\widehat{F}_{n,M}.
\]
We will formulate our subsequent result for joint asymptotics of $N$ and $M=M_N$, where $M$ can either be fixed or a diverging sequence $M_N \to \infty$ as $N \to \infty$. Since our theory requires separability, we embed the cdfs into a separable Banach space and choose the weighted $L^2$-functions on the real line. More precisely, letting $w$ be a continuous probability density that is positive everywhere on the real line, we define the inner product
\[
\langle f,g\rangle := \int_\mathbb{R} f(x) g(x) w(x) dx
\]
and the corresponding norm by $\|f\|^2 := \langle f,f\rangle $. The resulting space, which we will simply denote by $L^2(\mathbb{R})$, is a separable Hilbert space.
Next, we notice that the empirical cdfs are independent and satisfy the  exponential moment condition for any $\beta>1/2$
\begin{align} \label{e:errorF}
\varepsilon_{n,M} := \sqrt{M}[\widehat{F}_{n,M}-F_n], \qquad 
\sup_{n,M} \mathbb{E}\exp(d \|\varepsilon_{n,M}\|^{1/\beta})<\infty, \quad \forall d >0
\end{align}
which follows by the two-sided Dvoretzky–Kiefer–Wolfowitz inequality (for details see the proof of Theorem  \ref{theo3}).
The errors are independent, but non-stationary in the presence of changes, as their covariance function depends on $F_n$. To state convergence of the partial sum process, we now need to impose some regularity conditions on the functions $F_1,...,F_N$ - this will be needed to obtain a convergent covariance function. We choose a standard formulation of the problem, with a fixed number of $K$ change points and parameters 
\[
0=:\theta_0<\theta_1<...<\theta_K<\theta_{K+1}:=1\quad \textnormal{and} \quad c_k = 
 \lfloor N \theta_k \rfloor.
\]
We point out that, at the expense of a more involved presentation, it is also possible to modify the results from this section for scenarios where the number of changes $K=K_N$ diverges with $N$. For the sake of clarity we do not pursue such a regime here. We now formulate the main assumption of this section. 
\begin{assumption} \label{ass_3}
    There exist continuous cdfs $F^{(k,*)}$ $k=1,...,K+1$ such that 
    \[
   \lim_{N \to \infty }\max_{k=1,...,K+1} \sup_{x \in \mathbb{R}}|F_N^{(k)}(x)-F^{(k,*)}(x)| =0.
    \]
\end{assumption}
A setting where $F^{(k,*)} = F^{(j,*)}$ for all $k \neq j$ corresponds to local alternatives. Notice that here, the jump size depends on both $N$ (the convergence rate of $F^{(k)}_N$ to $F^{(k,*)}$) and $M$, since
\[\Delta_j := \sqrt{M}\|F_N^{(j)}-F_N^{(j-1)}\|.
\]
We now demonstrate that the convergence from Assumption \ref{ass_1} holds.
\begin{theo} \label{theo3}
{\it Suppose the model proposed in this section and suppose that Assumption \ref{ass_3} holds. Then, the partial sum process of errors $P_N$, defined in \eqref{e:lin}, based on the model errors \eqref{e:errorF} satisfies the weak convergence 
\[
    \{P_{N}(x): x \in [0,1]\}\overset{d}{\to}\{W(x): x \in [0,1]\}
\]
in the Hölder space $\mathcal{C}^\rho\big([0,1], L^2(\mathbb{R})\big)$ for any weight function $\rho$, polynomial or logarithmic, defined in \eqref{polyweights} and \eqref{logweights}. Here, $W$ is a Gaussian process with sample paths in $L^2(\mathbb{R})$ and governed by the covariance function
\begin{align*}
C^\varepsilon(x,y)[f,g]:= &\int_0^{\min(x,y)}\bigg\{ \int_\mathbb{R}\int_\mathbb{R} f(s) g(t) c^\varepsilon(s,t,z) w(s)w(t) ds dt \bigg\}dz,\qquad \textnormal{where}\\
c^\varepsilon(s,t,z):= & [F^{(k,*)}(\min(s,t))-F^{(k,*)}(s)F^{(k,*)}(t)]\cdot \mathbb{I}\{z \in [\theta_{k-1}, \theta_{K+1}]\}.
\end{align*}
}
\end{theo}
In practice the covariance operator $C^\varepsilon$ is unknown, but can be approximated by the following, natural estimator:
\begin{align} \label{e:hC}
\widehat{C}^\varepsilon(x,y)[f,g]:= \sum_{n=1}^{\lfloor N \min(x,y) \rfloor} \int_\mathbb{R} \int_\mathbb{R} f(s) g(t) [\widehat{F}_{n,M}(\min(s,t))-\widehat{F}_{n,M}(s)\widehat{F}_{n,M}(t)] ds dt.
\end{align}
Theorem \ref{theo3} demonstrates a weak convergence result of the partial sum process in the Hölder space. For the model considered in this section, this validates Assumption \ref{ass_1} (used in Theorems \ref{theo1} and \ref{theo2}), thus justifying the use of Algorithm \ref{alg1} for multiple change point detection in the distribution functions. 

\section{Monte Carlo Simulations} \label{sec:sim}
This section assesses the finite sample performance of the proposed multi-scale change point detection procedure via Monte Carlo simulation. We first consider an i.i.d. data generating process (DGP), followed by a dependent DGP specified by a functional autoregressive model, and finally the distribution functions of a data panel.

% Common setting
Several implementation choices are required for the simulation study. Data are modeled on a space of square integrable random functions, and if not specified otherwise, supported on the unit interval.
Our procedure depends on the weight function $\rho$, for which we consider the polynomial weights $\rho^{poly, \beta} (x)=x^\beta$ with $\beta=0.25$, and the logarithmic weights $\rho^{log, \beta}(x)=x^{1/2}\log^\beta(x^{-1})$ with $\beta=1$. The MultiScan algorithm further requires an input index set $\mathbf{N} \subset \mathbf{N}^{all}$. We examine both the set of all admissible tuples  $\mathbf{N}^{all}$  and the thinned out version $\mathbf{N}_N^{\theta} = \left\lbrace (n,h) \in \mathbf{N}^{all}: \exists m \in \mathbb{N} \;s.t.\; { h =  \lfloor 1.1^m} \rfloor\right\rbrace $. This design allows us to assess the effect of thinning out on detection accuracy and computational efficiency. All results are based on $R=1,000$ Monte Carlo repetitions. Within each repetition, the critical threshold $q$ is obtained from $B=1,000$ bootstrap samples.

% Performance measures
The proposed procedure is designed to detect and localize multiple change points. To evaluate its finite sample performance, we report the empirical rates of successful weak and strong localization across Monte Carlo replications, which are measured as the sums of
\begin{itemize} 
	\item[(i)] \textit{Weak localization:} 
	\begin{align}\label{eq:emp_weak_loc}
		\frac{1}{R} \mathbb{I}\big(~\forall (n,h) \in \mathbf{C} \,\,\exists k \in \{1,...,K\}: n-h+1 \le  c_k \le n+h~\big),
	\end{align}
	where $\mathbb{I}(\cdot)$ is an indicator function, and $R$ is the Monte Carlo repetition times;
	\item[(ii)] \textit{Strong localization:} 
	\begin{align}\label{eq:emp_strong_loc} 
		\frac{1}{R} \mathbb{I}\big(&~\forall (n,h) \in \mathbf{C} \,\,\exists k \in \{1,...,K\}: n-h+1 \le  c_k \le n+h; \\
		&~\forall k \in \{1,...,K\} ~\exists (n,h) \in \mathbf{C} : n-h+1 \le  c_k \le n+h, |\mathbf{C}|=K\big).\nonumber  
	\end{align}
	
\end{itemize}

Our procedure can also be viewed within the conventional hypothesis testing framework. The null hypothesis corresponds to no change, $H_0: K=0$, against the alternative of at least one change, $H_A: K \ge 1$. We reject $H_0$ if the output set $\mathbf{C}$ from Algorithm 1 is non-empty  (i.e. $|\mathbf{C}|>0$). In order to be comparable with conventional hypothesis testing, we also report the empirical size under $H_0$ and the empirical power under $H_A$.

\subsection{Independent functional time series}\label{sec:ind}
Under the i.i.d. DGP, the error process $\{\varepsilon_n\}_{n=1}^N$ in Model \eqref{e:mod2} is specified as
$$
\varepsilon_n (\tau) = \sum_{m=1}^{M} c_{n, m} \phi_m (\tau),  \;\; \tau \in [0,1], \;n=1,...,N, 
$$
where $c_{n,m} \sim i.i.d. \; \mathcal{N}(0, 0.1^2)$, and $\phi_m (\tau), m=1,...,M$ are the B-spline basis functions. We use $M=13$ basis functions, corresponding to order-four B-splines with nine equally spaced interior knots on $[0,1]$.
Under the null (when $K=0$), the mean function in Model \eqref{e:mod2} is given by
$
\mu^{(1)} (\tau) = 0.
$
Under the alternative, we consider four mean structures that differ in the number and configuration of change points:
\begin{itemize}
	\item $H_{A,1}$ (when $K=1$): One change,  a single mean shift,
	$$
	\mu^{(1)} (\tau) = 0, \; \mu^{(2)} (\tau) = 0.05 \mbox{ and } c_1 = \lfloor 0.5N \rfloor.
	$$
	\item $H_{A,2}$ (when $K=2$): Two changes with distinct functional forms,
	\begin{align*}
		\mu^{(1)} (\tau) = 0, &\; \mu^{(2)} (\tau) = 0.05, \; \mu^{(3)} (\tau) = 0.1 \sin(2 \pi \tau), \\
		\text{and }  & c_1 = \lfloor 0.3N \rfloor, c_2 = \lfloor 0.7N \rfloor.
	\end{align*}
	%		\item $H_{A,3}$ (when $K=3$):
	%		\begin{align*}
		%			& \mu^{(1)} (\tau) = 0, \; \mu^{(2)} (\tau) = 0.05, \; \mu^{(3)} (\tau) = 0.1 \sin(2 \pi \tau), \;  \mu^{(4)} (\tau) = 0.1 \cos(2 \pi \tau) \\
		%			\text{and }  & c_1 = \lfloor 0.3N \rfloor, c_2 = \lfloor 0.6N \rfloor , c_3 = \lfloor 0.8N \rfloor.
		%		\end{align*}
	%		This is the case where we have different distances of change points to its neighbors.
	\item $H_{A,3}$ (when $K=3$): Three unequally spaced changes,
	\begin{align*}
		& \mu^{(1)} (\tau) = 0, \; \mu^{(2)} (\tau) = 0.05, \; \mu^{(3)} (\tau) = 0, \;  \mu^{(4)} (\tau) = 0.1 \sin(2 \pi \tau) \\
		\text{and }  & c_1 = \lfloor 0.3N \rfloor, c_2 = \lfloor 0.6N \rfloor , c_3 = \lfloor 0.8N \rfloor.
	\end{align*}
	This scenario allows us to investigate the effects of $\delta_k$ (distance between neighboring changes) and $\Delta_k$ (magnitude of change). The first two changes differ in $\Delta_k$ but share the same $\delta_k$, while the last two have the same $\Delta_k$ but differ in $\delta_k$.
	\item $H_{A,4}$ (when $K=5$): Multiple closely spaced changes with varied functional forms,
	\begin{align*}
		&\mu^{(1)} (\tau) = 0, \; \mu^{(2)} (\tau) = 0.05, \; \mu^{(3)} (\tau) = 0.1 \sin(2 \pi \tau), \;  \mu^{(4)} (\tau) = 0.1 \cos(2 \pi \tau),   \\
		& \mu^{(5)} (\tau) = -0.1 + 0.2\tau,  \; \mu^{(6)} (\tau) = 0.8(\tau-0.5)^2 - 0.1 \\
		\text{and }  & c_1 = \lfloor 0.2N \rfloor, c_2 = \lfloor 0.4N \rfloor , c_3 = \lfloor 0.6N \rfloor, c_4 = \lfloor 0.7N \rfloor , c_5 = \lfloor 0.9N \rfloor.
	\end{align*}
	This complex scenario with five changes showcases the advantage of the proposed multiscale approach in detecting multiple changes of varying form and magnitude.
\end{itemize}

%Based on the above DGPs, we first generated the functional data and then save them as discretized vectors of evaluations. Thus, the function $f: [0,1] \to \mathbb{R}$ becomes the vector $(f(x_1), \dots, f(x_D))$, where $(x_d)_{d=1,\dots,D}$ is an equally-spaced grid on the unit interval.
%$$
%\widehat{C} (x_d,x_e) = \frac{1}{2(N-1)} \sum_{n=2}^{N} \left[ X_n (x_d)-X_{n-1}(x_d)\right] \left[ X_n (x_e) - X_{n-1}(x_e)\right] ,  \qquad d, e = 1,\dots,D.
%$$

To implement the proposed procedure, it is required to estimate the covariance of the errors, in accordance with \eqref{det11}. For the independent case, we use an estimator based on first-order differences of the data
$$
\widehat{C} = \frac{1}{2(N-1)} \sum_{n=2}^{N} \left( X_n-X_{n-1}\right) \otimes \left( X_n  - X_{n-1}\right).
$$

The next crucial step is to obtain the critical threshold $q$. For this purpose, we employ a bootstrap procedure for functional time series, detailed in Algorithm \ref{alg2}. 
This algorithm generates resampled observations based on the estimated covariance $\widehat{C}$, calculate the bootstrapped statistic ${L}^{\rho}$, and determine the critical threshold $q$ as the $(1-\alpha)^{th}$ quantile of the empirical distribution of the bootstrapped $\tilde{L}^{\rho}$. Finally, we apply Algorithm  \ref{alg1} with the critical threshold $q$ to obtain the output of a set of detected changes $\mathbf{C}$. The entire simulation process (from generating data, calculating the covariance matrix, bootstrapping the critical threshold, and obtaining a set of detected changes) is repeated $R=1,000$ times to evaluate the empirical performance of the proposed procedure.

\begin{algorithm}
	\caption{Bootstrapping procedure to obtain critical threshold for functional time series} \label{alg2}
	\begin{algorithmic}[0]
		\State \textbf{Input:} $\widehat{C}$, $\mathbf{N} \in \mathbf{N}^{all}$, $B$, $\alpha$ 
		\State \textbf{Output:} $q$
		\Function{Boot}{$\widehat{C}$} 
		\State Perform eigenvalue decomposition on $\widehat{C}$ by $\widehat{C} = Q \cdot \Lambda  \cdot Q^\T$
		%\State \parbox[t]{\dimexpr\linewidth-\algorithmicindent}{\hangindent=1.5em For the eigenvalues $\lambda_i, i=1,...,D$ in $\Lambda$, set it as zero if it is very close to zero (e.g. $\left| \lambda_i \right| < 1e-08$) . The new diagonal matrix of eigenvalues is denoted as $\Lambda^*$}
		\State Take the square root of $\Lambda$ and get  ${\Lambda}_{root}$
		\State Obtain the square root of the matrix $\widehat{C}_{root} = Q \cdot {\Lambda^{*}}_{root} \cdot Q^\T$
		\State \parbox[t]{\dimexpr\linewidth-\algorithmicindent}{\hangindent=1.5em  Generate random vectors,  $Z_1, ..., Z_N \sim \mathcal{N}({0}, {I_D})$, where ${I_D}$ is the $D$-variate identity matrix}
		\State \parbox[t]{\dimexpr\linewidth-\algorithmicindent}{\hangindent=1.5em  Compute $\tilde{\epsilon}_n = \widehat{C}_{root} \cdot Z_n $ for $n=1,...,N$}
		\State Calculate 
		$$
		\tilde{L}^{\rho} = \sup_{(n,h)\in \mathbf{N}} \dfrac{\left\| \sum_{\ell=n-h-1}^{n} \tilde{\varepsilon}_n - \sum_{\ell=n+1}^{n+h} \tilde{\varepsilon}_n \right\| }{N^{1/2}\rho(h/N)}
		$$   
		\Return $\tilde{L}^{\rho}$
		\EndFunction
		\State \textbf{Repeat:} Call BOOT($\widehat{C}$) for $B$ times to   to obtain $\{\tilde{L}^{\rho}_{(b)}\}_{b=1}^B$.
		\State \textbf{Set:} $q$ as the empirical $(1-\alpha)$-quantile of $\{\tilde{L}^{\rho}_{(b)}\}_{b=1}^B$.
	\end{algorithmic}
\end{algorithm}
 
Table \ref{tab:H0_size} reports the empirical size under the null for two versions of the input set $\mathbf{N}$ and two choices of weight function $\rho$. The empirical size is observed to converge to the nominal significance level $\alpha$ as $N$ increases, confirming the asymptotic validity in Proposition \ref{prop2}. Comparing between the two weight functions, the size based on the polynomial weights is more undersized when $N=100$, whereas the logarithmic weights provide better size control in small samples.  Notably, the empirical sizes are close when using the full set $\mathbf{N}^{all}$ versus the thinned out version $\mathbf{N}_N^{\theta}$. This indicates that the computationally efficient $\mathbf{N}_N^{\theta}$ input set can be employed without a compromise to the empirical size.

\begin{table}[htbp]
	\centering
	\caption{Empirical size based on the i.i.d. DGP under the null}    
	\begin{tabular}{lrlllllll}
		\toprule
		\toprule
		&       & \multicolumn{3}{c}{$\mathbf{N}^{all}$} &       & \multicolumn{3}{c}{$\mathbf{N}_N^{\theta}$} \\
		&       & 10\%  & 5\%   & 1\%   &       & 10\%  & 5\%   & 1\% \\
		\cmidrule{3-9}    \underline{Polynomial weights} &       &       &       &       &       &       &       &  \\
		$N=100$ &       & 0.078 & 0.035 & 0.011 &       & 0.076 & 0.042 & 0.009 \\
		$N=200$ &       & 0.079 & 0.035 & 0.008 &       & 0.068 & 0.035 & 0.007 \\
		$N=300$ &       & 0.095 & 0.050 & 0.009 &       & 0.093 & 0.047 & 0.006 \\
		&       &       &       &       &       &       &       &  \\
		\underline{Logarithmic weights} &       &       &       &       &       &       &       &  \\
		$N=100$ &       & 0.079 & 0.042 & 0.009 &       & 0.083 & 0.044 & 0.011 \\
		$N=200$ &       & 0.088 & 0.045 & 0.009 &       & 0.092 & 0.028 & 0.006 \\
		$N=300$ &       & 0.104 & 0.055 & 0.010 &       & 0.098 & 0.045 & 0.009 \\
		\bottomrule
		\bottomrule
	\end{tabular}%
	\label{tab:H0_size}%
\end{table}%

To conserve space, the results under the alternatives are reported for the polynomial weights.\footnote{Results based on the logarithmic weights are  similar.} Table \ref{tab:HA_power_IID} presents the empirical performance under the alternative in terms of three measures: empirical power, weak localization, and strong localization. The empirical power and weak localization rates are high across all settings, even for small sample sizes ($N=100$). Strong localization is also effective under $H_{A,1}$ and $H_{A,2}$. For the more complex $H_{A,3}$ and $H_{A,4}$ settings, which involve more changes or smaller separations, larger sample sizes ($N \ge 200$) are necessary to achieve good strong localization rates. Again, the differences between $\mathbf{N}^{all}$ and $\mathbf{N}_N^{\theta}$ are negligible, indicating that the thinned out version is computationally efficient and also effective under the alternatives. Figure \ref{fig:cploc_iid} plots the empirical distribution of the detected locations $n_{(n,h)}^{max}$ from a representative simulation under all four alternatives with $N=300$.

\begin{table}[htbp]
	\centering
	\caption{Empirical performance based on the i.i.d. DGP under the alternative}
	\begin{tabular}{lrccccccc}
		\toprule
		\toprule
		&       & \multicolumn{3}{c}{$\mathbf{N}^{all}$} &       & \multicolumn{3}{c}{$\mathbf{N}_N^{\theta}$} \\
		&       & \multicolumn{1}{l}{Power} & \multicolumn{1}{l}{Weak Loc.} & \multicolumn{1}{l}{Strong Loc.} &       & \multicolumn{1}{l}{Power} & \multicolumn{1}{l}{Weak Loc.} & \multicolumn{1}{l}{Strong Loc.} \\
		\cmidrule{3-9}    \underline{$H_{A,1}$} &       &       &       &       &       &       &       &  \\
		$N=100$ &       & 1.000 & 0.999 & 0.999 &       & 1.000 & 1.000 & 1.000 \\
		$N=200$ &       & 1.000 & 0.999 & 1.000 &       & 1.000 & 0.999 & 1.000 \\
		$N=300$ &       & 1.000 & 1.000 & 1.000 &       & 1.000 & 0.998 & 1.000 \\
		&       &       &       &       &       &       &       &  \\
		\underline{$H_{A,2}$} &       &       &       &       &       &       &       &  \\
		$N=100$ &       & 1.000 & 1.000 & 0.976 &       & 1.000 & 1.000 & 0.961 \\
		$N=200$ &       & 1.000 & 1.000 & 0.999 &       & 1.000 & 1.000 & 1.000 \\
		$N=300$ &       & 1.000 & 1.000 & 1.000 &       & 1.000 & 1.000 & 0.999 \\
		&       &       &       &       &       &       &       &  \\
		\underline{$H_{A,3}$} &       &       &       &       &       &       &       &  \\
		$N=100$ &       & 1.000 & 1.000 & 0.063 &       & 1.000 & 1.000 & 0.078 \\
		$N=200$ &       & 1.000 & 1.000 & 0.824 &       & 1.000 & 1.000 & 0.813 \\
		$N=300$ &       & 1.000 & 1.000 & 0.993 &       & 1.000 & 1.000 & 0.992 \\
		&       &       &       &       &       &       &       &  \\
		\underline{$H_{A,4}$} &       &       &       &       &       &       &       &  \\
		$N=100$ &       & 1.000 & 1.000 & 0.004 &       & 1.000 & 1.000 & 0.002 \\
		$N=200$ &       & 1.000 & 1.000 & 0.624 &       & 1.000 & 1.000 & 0.709 \\
		$N=300$ &       & 1.000 & 1.000 & 0.990 &       & 1.000 & 1.000 & 0.994 \\
		\bottomrule
		\bottomrule
	\end{tabular}%
	\label{tab:HA_power_IID}%
\end{table}%

\begin{figure}
	\centering
	\includegraphics[width=0.45\linewidth]{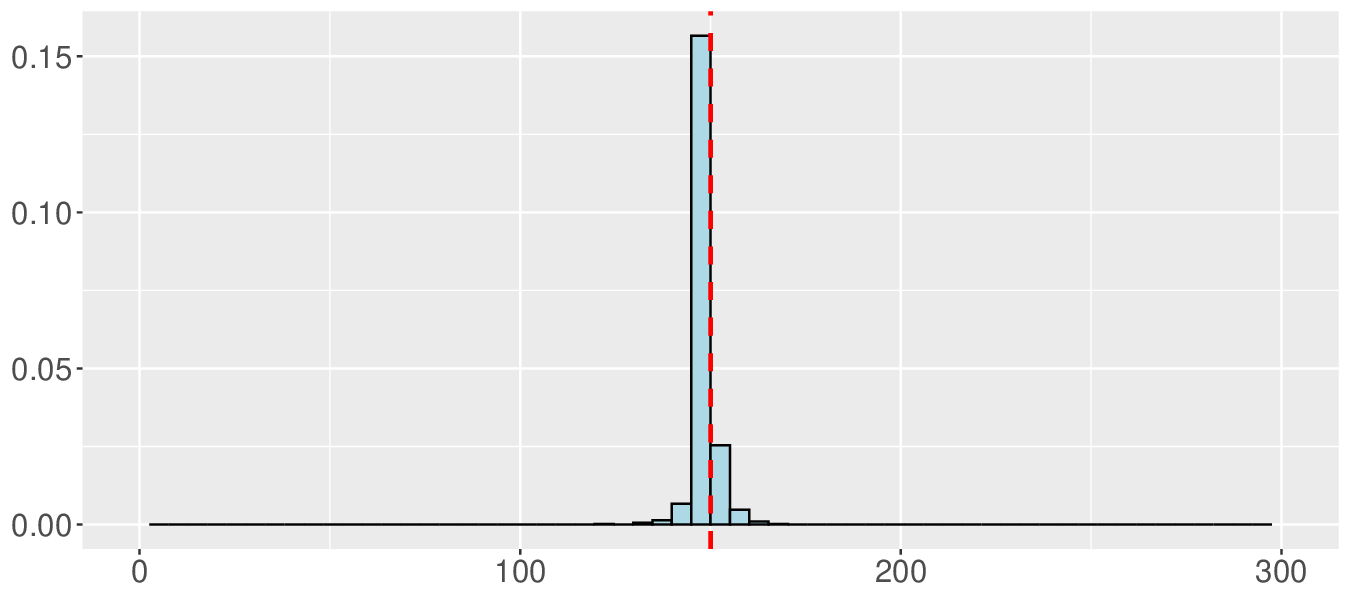}
	\includegraphics[width=0.45\linewidth]{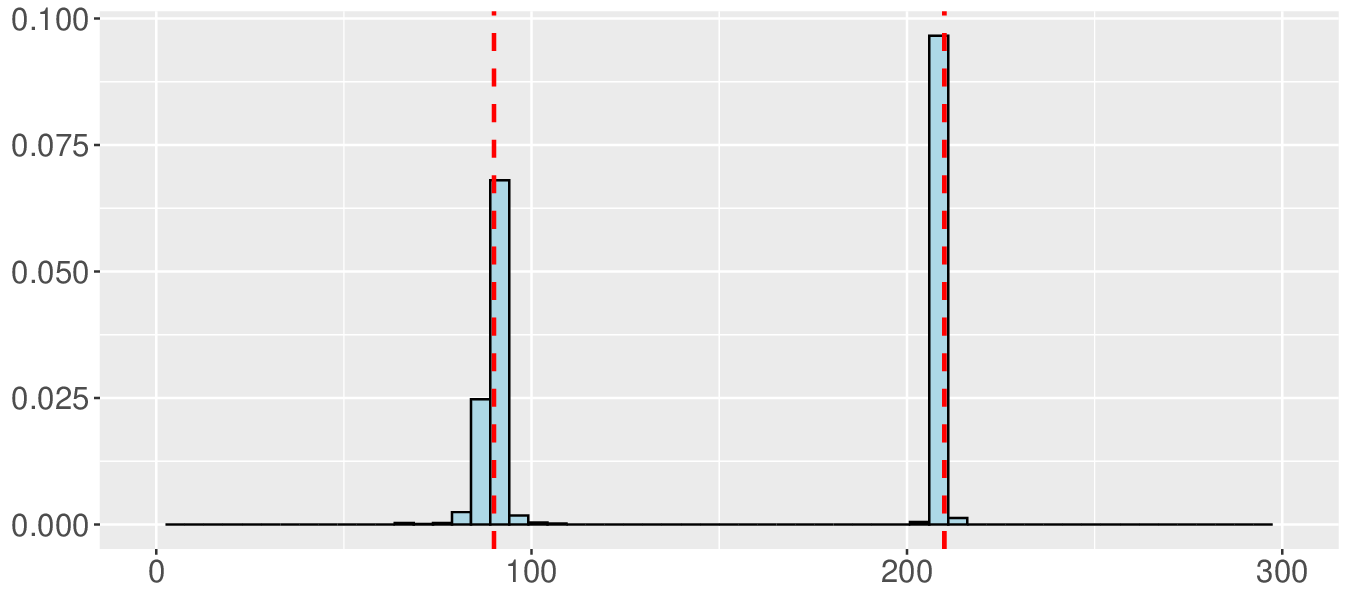}
	\includegraphics[width=0.45\linewidth]{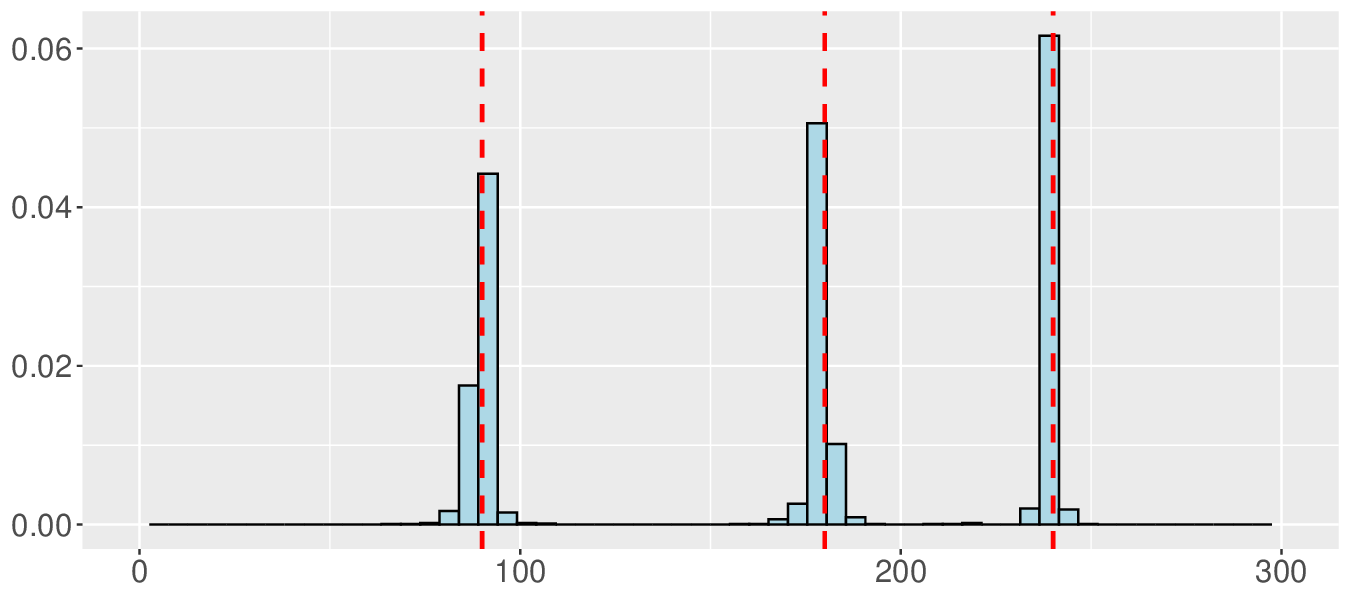}
	\includegraphics[width=0.45\linewidth]{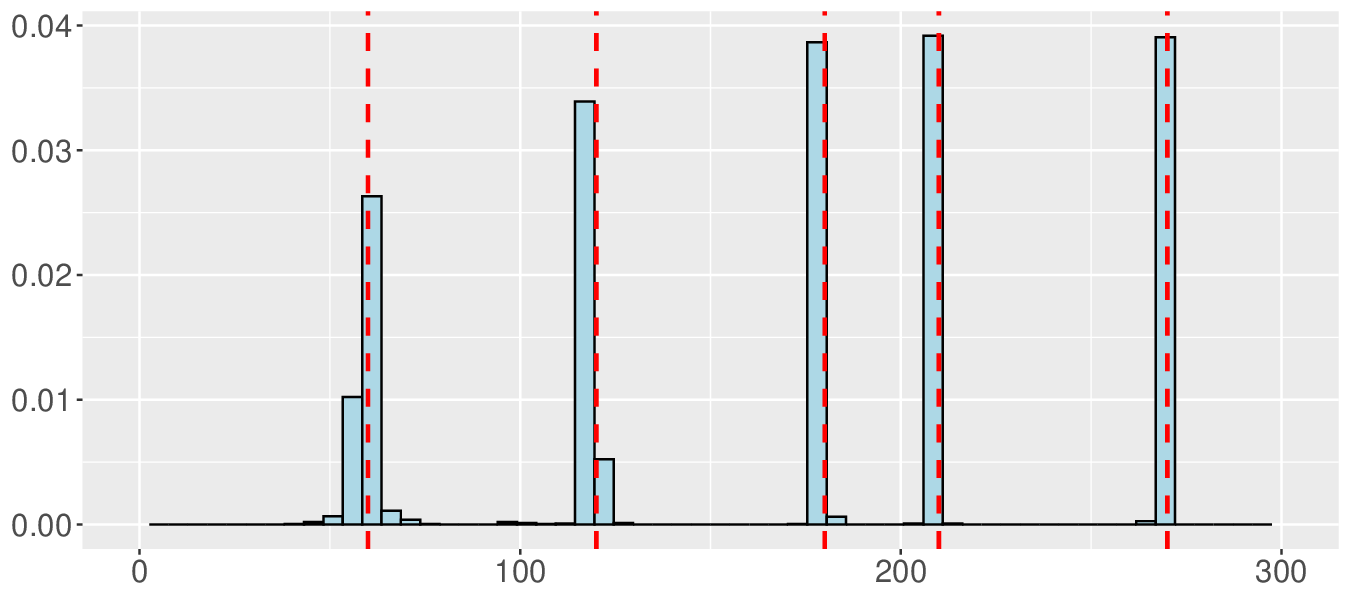}
	\caption{Frequency of $n_{(n,h)}^{max}$ with $N=300$ for the i.i.d. error. Top left: $H_{A,1}$; Top right: $H_{A,2}$; Bottom left: $H_{A,3}$; Bottom right: $H_{A,4}$.}
	\label{fig:cploc_iid}
\end{figure}

\subsection{Dependent functional time series}\label{sec:dep}
In the dependent DGP, the mean functions are the same to those specified in the independent case, but the error process $\{\varepsilon_n\}_{n=1}^N$ follows a functional autoregressive (FAR) model,
$$
	\varepsilon_n(\tau)   = \int \psi(\tau, s) e_{n-1}(s) d s + e_n (\tau),  \;\; \tau \in [0,1], \; n=1,...,N,  
$$
where the innovation process is
$$
	e_n (\tau) = \sum_{m=1}^{M} c_{n, m} \phi_m (\tau), 
$$
and $c_{n,m} \sim i.i.d. \; \mathcal{N}(0, 0.1^2)$, and $\{\phi_m\}$ being the same B-spline basis as before. The kernel function is specified as $\psi(\tau, s) = 1/4 s \tau$.

The simulation for the dependent case is identical to that in Section \ref{sec:ind}, except that the covariance estimator is replaced by an asymptotically consistent estimator of the long-run covariance for weakly dependent errors. \cite{carlstein1986use} proposed a subseries variance estimate for strong mixing processes by using non-overlapping blocking technique, which was later extended by \cite{wu:zhao:2007} for non-stationary time series. To adapt this approach for functional time series, we propose to use
\begin{equation}\label{eq:WZ_est}
	\hat C = \frac{1}{2(N/k_N-1)}\sum_{i=2}^{k_N} (A_i-A_{i-1})\otimes (A_i-A_{i-1}),
\end{equation}
where 
$$ A_i = k_N^{-1/2} \sum_{j=(i-1)k_N+1}^{ik_N} X_j, \qquad i=1,\dots, M,$$ 
and $k_N$ is a bandwidth parameter defining the size of $M$ blocks. In our implementation, we set $k_N=3$.\footnote{Alternative choices, such as $K_N=5,...,10$, have similar empirical performance.}

Table \ref{tab:H0_Size_FAR} reports the empirical size under the null hypothesis for the DGP with FAR errors, and Table \ref{tab:HA_Power_FAR} summarizes the corresponding results under the alternative. The test exhibits good size control and high power. Weak localization performs well even for small samples, while strong localization requires larger sample sizes to achieve comparable performance. Figure \ref{fig:cploc_far} shows the frequency of $n_{(n,h)}^{max}$ from a representative Monte Carlo repetition with $N=300$ under the four alternatives.

% Table generated by Excel2LaTeX from sheet 'FAR'
\begin{table}[htbp]
	\centering
	\caption{Empirical size based on the dependent DGP under the null}
	\begin{tabular}{lrccccccc}
		\toprule
		\toprule
		&       & \multicolumn{3}{c}{$\mathbf{N}^{all}$} &       & \multicolumn{3}{c}{$\mathbf{N}_N^{\theta}$} \\
		&       & 10\%  & 5\%   & 1\%   &       & 10\%  & 5\%   & 1\% \\
		\cmidrule{3-9}    \underline{Polynomial weight} &       &       &       &       &       &       &       &  \\
		$N=100$ &       & 0.051 & 0.023 & 0.004 &       & 0.045 & 0.024 & 0.005 \\
		$N=200$ &       & 0.091 & 0.042 & 0.008 &       & 0.094 & 0.036 & 0.007 \\
		$N=300$ &       & 0.084 & 0.040 & 0.010 &       & 0.080 & 0.043 & 0.015 \\
		&       &       &       &       &       &       &       &  \\
		\underline{Logarithmic weight} &       &       &       &       &       &       &       &  \\
		$N=100$ &       & 0.080 & 0.026 & 0.005 &       & 0.069 & 0.031 & 0.005 \\
		$N=200$ &       & 0.102 & 0.041 & 0.011 &       & 0.094 & 0.053 & 0.016 \\
		$N=300$ &       & 0.095 & 0.053 & 0.014 &       & 0.090 & 0.046 & 0.010 \\
		\bottomrule
		\bottomrule
	\end{tabular}%
	\label{tab:H0_Size_FAR}%
\end{table}%

% Table generated by Excel2LaTeX from sheet 'FAR'
\begin{table}[htbp]
	\centering
	\caption{Empirical performance based on the dependent DGP under the alternative}
	\begin{tabular}{lrccccccc}
		\toprule
		\toprule
		&       & \multicolumn{3}{c}{$\mathbf{N}^{all}$} &       & \multicolumn{3}{c}{$\mathbf{N}_N^{\theta}$} \\
		&       & Power & Weak Loc. & Strong Loc. &       & Power & Weak Loc. & Strong Loc. \\
		\cmidrule{3-9}    $\underline{H_{A,1}}$ &       &       &       &       &       &       &       &  \\
		$N=100$ &       & 1.000 & 1.000 & 1.000 &       & 1.000 & 1.000 & 1.000 \\
		$N=200$ &       & 1.000 & 0.999 & 1.000 &       & 1.000 & 1.000 & 1.000 \\
		$N=300$ &       & 1.000 & 1.000 & 1.000 &       & 1.000 & 0.999 & 1.000 \\
		&       &       &       &       &       &       &       &  \\
		$\underline{H_{A,2}}$ &       &       &       &       &       &       &       &  \\
		$N=100$ &       & 1.000 & 1.000 & 0.867 &       & 1.000 & 1.000 & 0.887 \\
		$N=200$ &       & 1.000 & 1.000 & 0.999 &       & 1.000 & 1.000 & 0.999 \\
		$N=300$ &       & 1.000 & 1.000 & 1.000 &       & 1.000 & 0.999 & 1.000 \\
		&       &       &       &       &       &       &       &  \\
		$\underline{H_{A,3}}$ &       &       &       &       &       &       &       &  \\
		$N=100$ &       & 1.000 & 1.000 & 0.047 &       & 1.000 & 1.000 & 0.040 \\
		$N=200$ &       & 1.000 & 1.000 & 0.616 &       & 1.000 & 1.000 & 0.651 \\
		$N=300$ &       & 1.000 & 1.000 & 0.987 &       & 1.000 & 1.000 & 0.975 \\
		&       &       &       &       &       &       &       &  \\
		$\underline{H_{A,4}}$ &       &       &       &       &       &       &       &  \\
		$N=100$ &       & 1.000 & 1.000 & 0.001 &       & 1.000 & 1.000 & 0.003 \\
		$N=200$ &       & 1.000 & 1.000 & 0.260 &       & 1.000 & 1.000 & 0.286 \\
		$N=300$ &       & 1.000 & 1.000 & 0.955 &       & 1.000 & 1.000 & 0.963 \\
		\bottomrule
		\bottomrule
	\end{tabular}%
	\label{tab:HA_Power_FAR}%
\end{table}%

\begin{figure}
	\centering
	\includegraphics[width=0.45\linewidth]{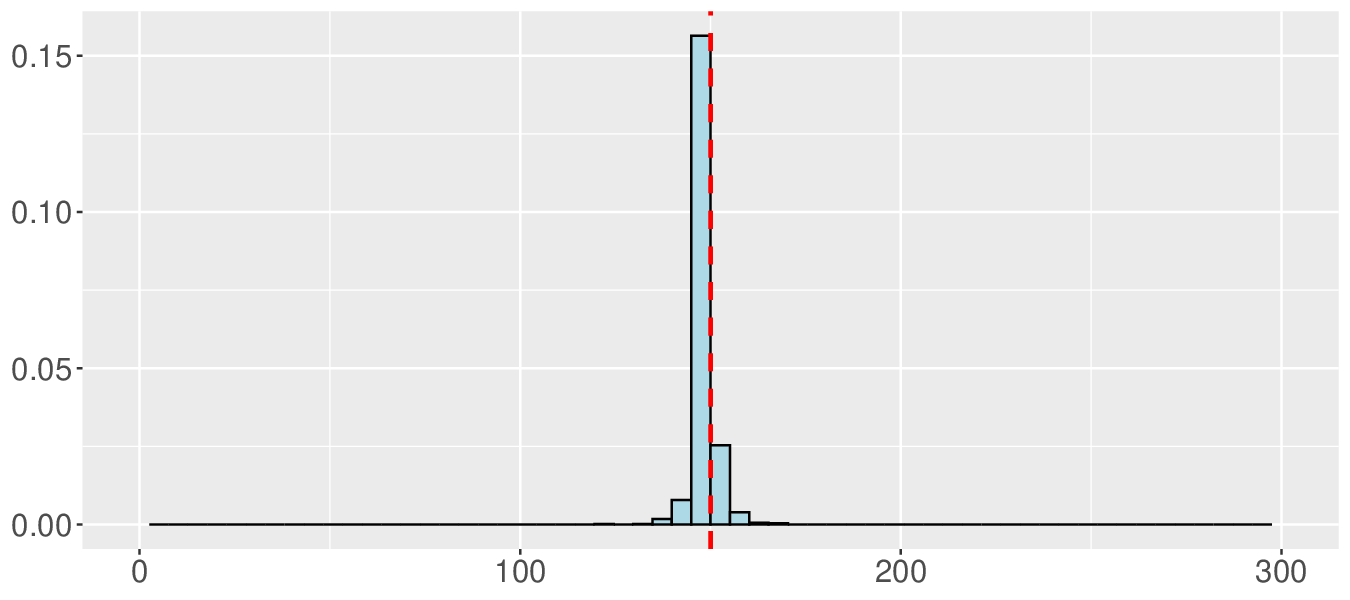}
	\includegraphics[width=0.45\linewidth]{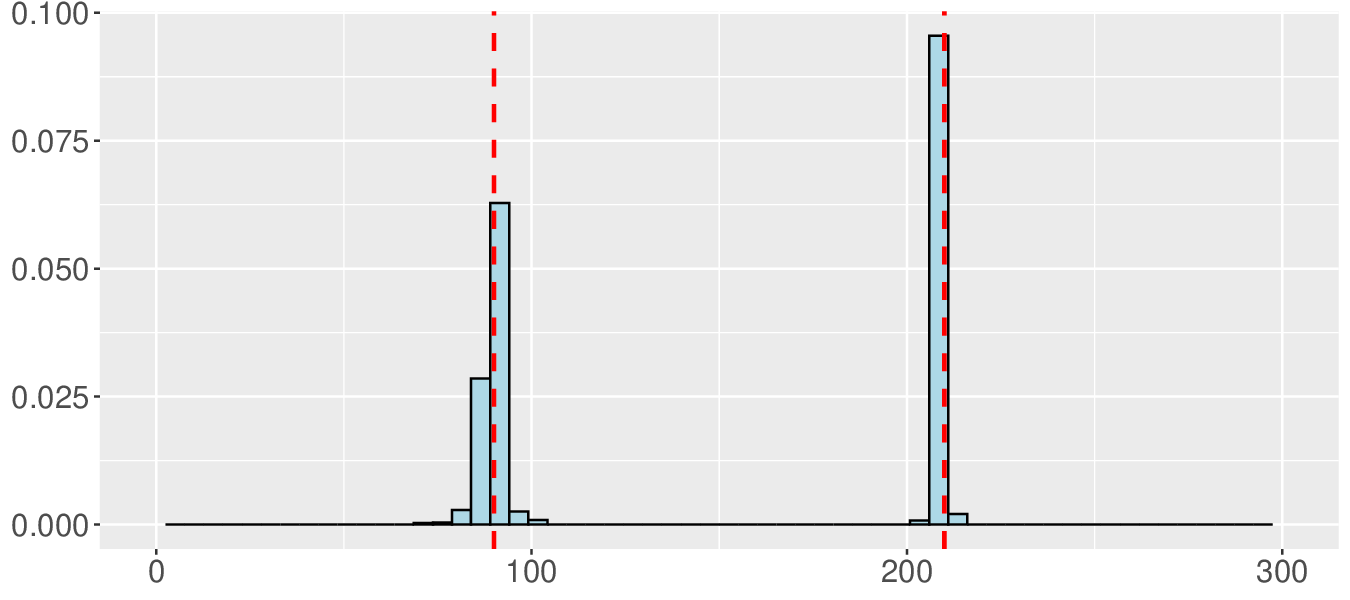}
	\includegraphics[width=0.45\linewidth]{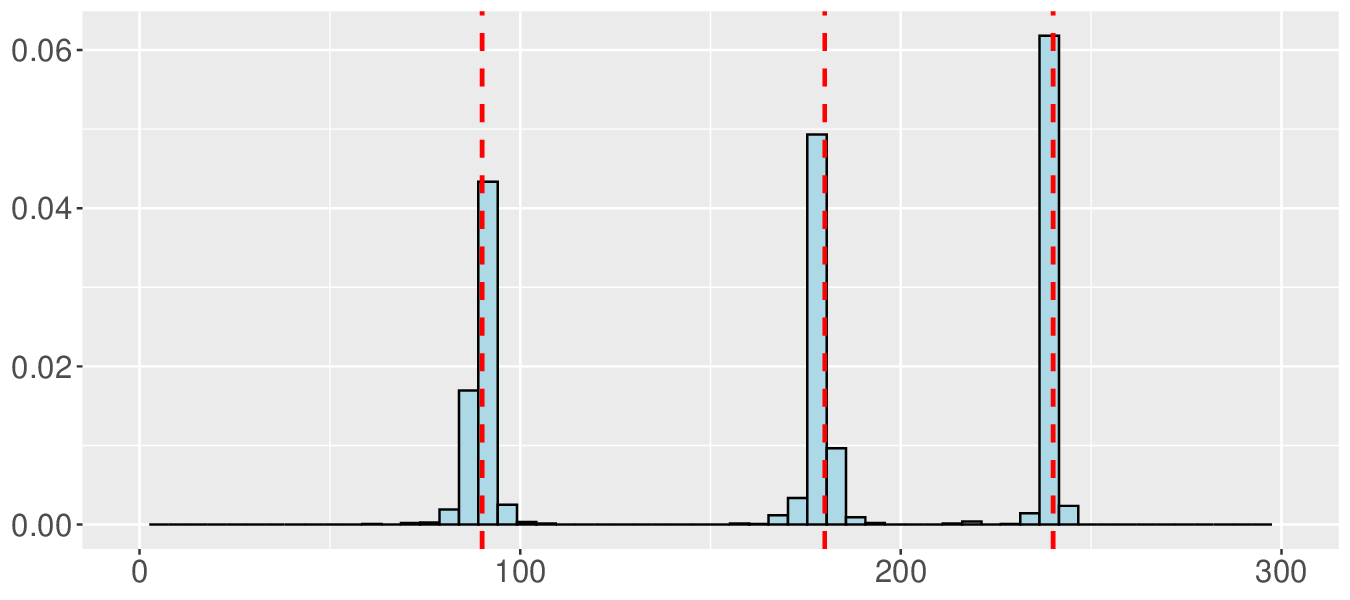}
	\includegraphics[width=0.45\linewidth]{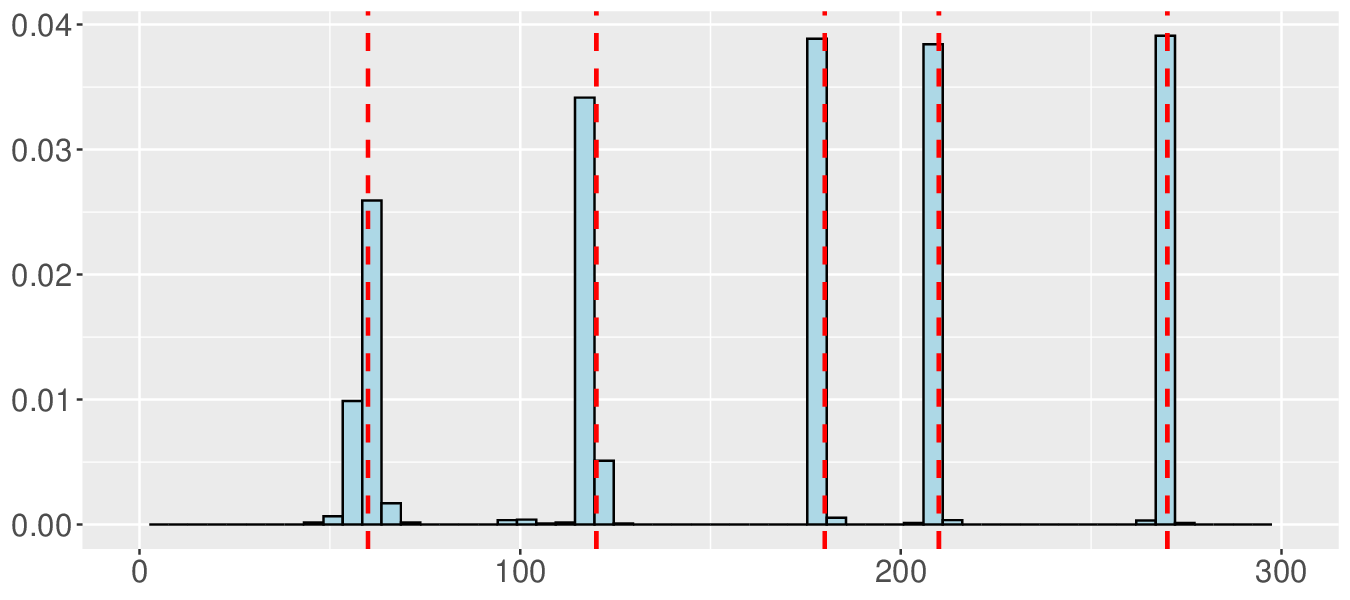}
	\caption{Frequency of $n_{(n,h)}^{max}$ with $N=300$ for the dependent DGP. Top left: $H_{A,1}$; Top right: $H_{A,2}$; Bottom left: $H_{A,3}$; Bottom right: $H_{A,4}$.}
	\label{fig:cploc_far}
\end{figure}

\subsection{Distribution functions in a data panel}
We generate distribution functions in a data panel under the null $H_{0}^*$ (when $K=0$) as
$$
F_n \sim \mathcal{N}(0,1) \;\; \mbox{ for } n=1,...,N.
$$
Under the alternative, we consider two scenarios:
\begin{itemize}
	\item $H_{A,1}^*$ (when $K=2$): Changes within the Gaussian family,
	\begin{align*}
		F^{(1)} \sim \mathcal{N}(0,1^2), &\; F^{(2)} \sim \mathcal{N}(0.05,1^2), \; F^{(3)} \sim \mathcal{N}(0.05,1.05^2), \\
		\text{and }  & c_1 = \lfloor 0.3N \rfloor, c_2 = \lfloor 0.7N \rfloor.
	\end{align*}
	This scenario introduces a change in the mean, followed by a change in the variance.
	%		\item $H_{A,3}$ (when $K=3$):
	%		\begin{align*}
		%			& \mu^{(1)} (\tau) = 0, \; \mu^{(2)} (\tau) = 0.05, \; \mu^{(3)} (\tau) = 0.1 \sin(2 \pi \tau), \;  \mu^{(4)} (\tau) = 0.1 \cos(2 \pi \tau) \\
		%			\text{and }  & c_1 = \lfloor 0.3N \rfloor, c_2 = \lfloor 0.6N \rfloor , c_3 = \lfloor 0.8N \rfloor.
		%		\end{align*}
	%		This is the case where we have different distances of change points to its neighbors.
	\item $H_{A,2}^*$ (when $K=2$): Changes between distributional families, 
	\begin{align*}
		F^{(1)} \sim \mathcal{N}(0,1^2), &\; F^{(2)} \sim t(df=10), \; F^{(3)} \sim \mbox{skewed } t(skew=0.05,df=10), \\
		\text{and }  & c_1 = \lfloor 0.3N \rfloor, c_2 = \lfloor 0.7N \rfloor.
	\end{align*}
	This scenario changes from a Gaussian to a Student's $t$-distribution, and subsequently to a skewed $t$-distribution.
\end{itemize}

After generating the data, we obtain the empirical cumulative distribution functions $\widehat{F}_n, n=1,\dots,N$.  As detailed in Section \ref{sec3}, the underlying error $\varepsilon_{n}$
is independent but non-stationary under the alternative, as the covariance structure of $\varepsilon_{n}$
is dependent on $F_n$. However, at each time point $n$, we can approximate the instantaneous covariance by its natural empirical counterpart,
$$
	\widehat{C}_n(s,t) = \widehat{F}_n(\min(s,t)) - \widehat{F}_n(s)\widehat{F}_n(t), \qquad s,t \in \mathbb{R}. 
$$ 

To obtain the critical threshold $q$, we propose to use Algorithm \ref{alg3} for distributions functions. In this algorithm, at each time point $n$, resampled observations are generated based on the estimated covariance $\widehat{C}_n$. The threshold $q$ is determined as the $(1-\alpha)$-quantile of the bootstrapped $\tilde{L}^{\rho}$.
Finally, the theoretical framework in Section \ref{sec3} requires a continuous weight function $w(x)$. For the simulation, we employ a flattop function with exponential tails:
$$
w(x) = \begin{cases}
	\exp(x+A), \;\; &x < -A,\\
	1, \;\; &x \in \left[-A, A\right], \\
	\exp(A-x), \;\; &x > A,\\
\end{cases}
$$
where $A$ is a parameter that we set to be 2.

\begin{algorithm}
	\caption{Bootstrapping procedure to obtain critical threshold for distribution functions} \label{alg3}
	\begin{algorithmic}[0]
		\State \textbf{Input:} $\left\lbrace \widehat{C}_n\right\rbrace_{n=1}^N $, $\mathbf{N} \in \mathbf{N}^{all}$, $B$, $\alpha$ 
		\State \textbf{Output:} $q$
		\Function{Boot-ECDF}{$\left\lbrace \widehat{C}_n\right\rbrace_{n=1}^N $} 
		\For{$n=1,\ldots,N$}
		\State Perform eigenvalue decomposition on $\widehat{C}_n$ by $\widehat{C}_n = Q_n \cdot \Lambda_n  \cdot Q_n^\T$
		%\State \parbox[t]{\dimexpr\linewidth-\algorithmicindent}{\hangindent=1.5em For the eigenvalues $\lambda_{i,n}, i=1,...,D$ in $\Lambda$, set it as zero if it is very close to zero (e.g. $\left| \lambda_{i,n} \right| < 1e-08$) . The new diagonal matrix of eigenvalues is denoted as $\Lambda_n^*$.}
		\State Take the square root of $\Lambda_n$ and get  ${\Lambda}_{n,root}$
		\State Obtain the square root of the matrix $\widehat{C}_{n,root} = Q_n \cdot {\Lambda^{*}}_{root} \cdot Q_n^\T$
		\State \parbox[t]{\dimexpr\linewidth-\algorithmicindent}{\hangindent=1.5em  Generate a random vector  $Z_n \sim \mathcal{N}({0}, {Id})$, where ${Id}$ is the $D$-variate identity matrix}
		\State \parbox[t]{\dimexpr\linewidth-\algorithmicindent}{\hangindent=1.5em  Compute $\tilde{\epsilon}_n = \widehat{C}_{n,root} \cdot Z_n $}
		\EndFor
		
		\State Calculate 
		$$
		\tilde{L}^{\rho} = \sup_{(n,h)\in \mathbf{N}} \dfrac{\left\| \sum_{\ell=n-h-1}^{n} \tilde{\varepsilon}_n - \sum_{\ell=n+1}^{n+h} \tilde{\varepsilon}_n \right\| }{N^{1/2}\rho(h/N)}
		$$   
		\Return $\tilde{L}^{\rho}$
		\EndFunction
		\State \textbf{Repeat:} BOOT-ECDF($\left\lbrace \widehat{C}_n\right\rbrace_{n=1}^N $)  for $B$ times to obtain $\{\tilde{L}^{\rho}_{(b)}\}_{b=1}^B$.
		\State \textbf{Set:} $q$ as the empirical $(1-\alpha)$-quantile of $\{\tilde{L}^{\rho}_{(b)}\}_{b=1}^B$.
	\end{algorithmic}
\end{algorithm}

The empirical size is reported in Table \ref{tab:H0_Dist}, and the performance under the alternatives is summarized in Table \ref{tab:HA_Dist}.  The procedure exhibits good size control under $H_0^*$ and high power under $H_{A,1}^*$ and $H_{A,2}^*$. Consistent with the previous DGPs, weak localization rates are high even for small $N$. In contrast, strong localization requires larger sample sizes to achieve satisfactory level. Finally, Figure \ref{fig:cploc_dist} plots the frequency of $n_{(n,h)}^{max}$ for a representative simulation ($N=300$) under both alternative scenarios.

% Table generated by Excel2LaTeX from sheet 'IID Dist'
\begin{table}[htbp]
	\centering
	\caption{Empirical size based on distribution functions under the null}
	\begin{tabular}{lrlllllll}
		\toprule
		\toprule
		&       & \multicolumn{3}{c}{$\mathbf{N}^{all}$} &       & \multicolumn{3}{c}{$\mathbf{N}_N^{\theta}$} \\
		&       & 10\%  & 5\%   & 1\%   &       & 10\%  & 5\%   & 1\% \\
		\cmidrule{3-9}    \underline{Polynomial weight} &       &       &       &       &       &       &       &  \\
		$N=100$ &       & 0.090 & 0.046 & 0.013 &       & 0.099 & 0.048 & 0.011 \\
		$N=200$ &       & 0.090 & 0.046 & 0.007 &       & 0.104 & 0.051 & 0.016 \\
		$N=300$ &       & 0.096 & 0.053 & 0.014 &       & 0.114 & 0.066 & 0.013 \\
		&       &       &       &       &       &       &       &  \\
		\underline{Logarithmic weight} &       &       &       &       &       &       &       &  \\
		$N=100$ &       & 0.101 & 0.048 & 0.009 &       & 0.094 & 0.044 & 0.012 \\
		$N=200$ &       & 0.105 & 0.060 & 0.008 &       & 0.100 & 0.054 & 0.011 \\
		$N=300$ &       & 0.110 & 0.058 & 0.016 &       & 0.103 & 0.045 & 0.011 \\
		\bottomrule
		\bottomrule
	\end{tabular}%
	\label{tab:H0_Dist}%
\end{table}%

% Table generated by Excel2LaTeX from sheet 'Dist'
\begin{table}[htbp]
	\centering
	\caption{Empirical performance based on distribution functions under the alternative}
	\begin{tabular}{lrccccccc}
		\toprule
		\toprule
		&       & \multicolumn{3}{c}{$\mathbf{N}^{all}$} &       & \multicolumn{3}{c}{$\mathbf{N}_N^{\theta}$} \\
		&       & \multicolumn{1}{l}{Power} & \multicolumn{1}{l}{Weak Loc.} & \multicolumn{1}{l}{Strong Loc.} &       & \multicolumn{1}{l}{Power} & \multicolumn{1}{l}{Weak Loc.} & \multicolumn{1}{l}{Strong Loc.} \\
		\cmidrule{3-9}    $H_{A,1}$ &       &       &       &       &       &       &       &  \\
		$N=100$ &       & 0.993 & 0.992 & 0.187 &       & 0.991 & 0.987 & 0.181 \\
		$N=200$ &       & 1.000 & 0.998 & 0.745 &       & 1.000 & 0.998 & 0.729 \\
		$N=300$ &       & 1.000 & 1.000 & 0.980 &       & 1.000 & 1.000 & 0.984 \\
		&       &       &       &       &       &       &       &  \\
		$H_{A,2}$ &       &       &       &       &       &       &       &  \\
		$N=100$ &       & 1.000 & 0.999 & 0.790 &       & 1.000 & 0.995 & 0.803 \\
		$N=200$ &       & 1.000 & 0.997 & 0.997 &       & 1.000 & 0.999 & 0.999 \\
		$N=300$ &       & 1.000 & 0.999 & 1.000 &       & 1.000 & 0.999 & 1.000 \\
		\bottomrule
		\bottomrule
	\end{tabular}%
	\label{tab:HA_Dist}%
\end{table}%

\begin{figure}
	\centering
	\includegraphics[width=0.8\linewidth]{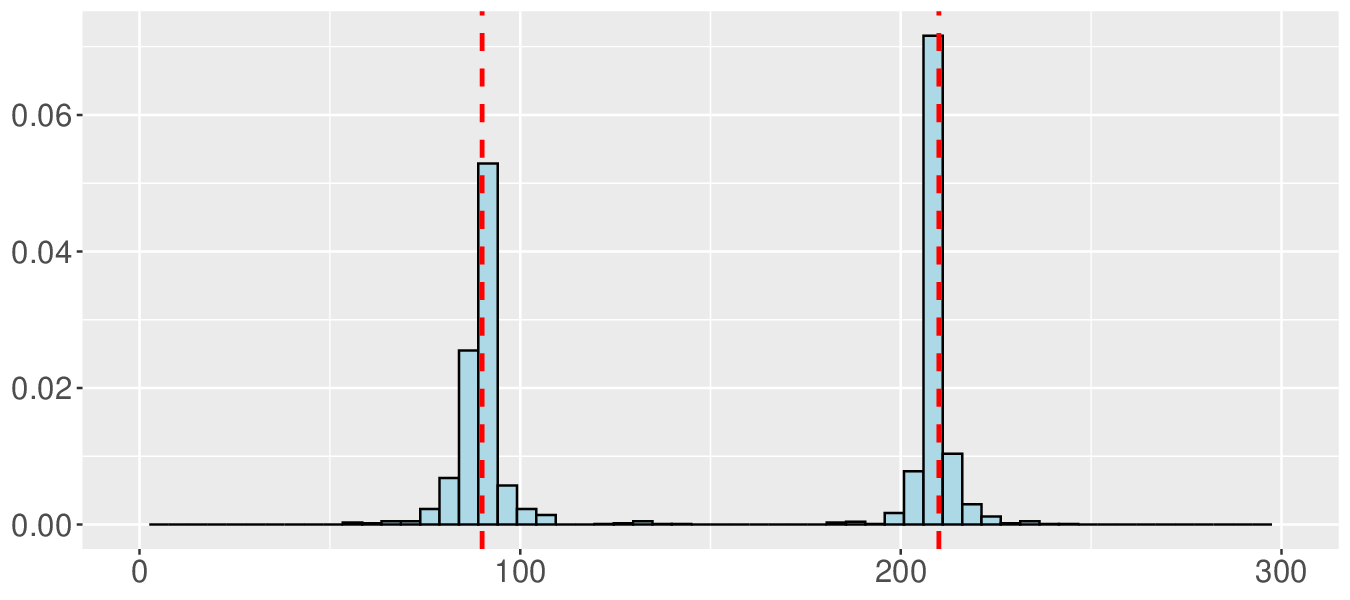}
	\includegraphics[width=0.8\linewidth]{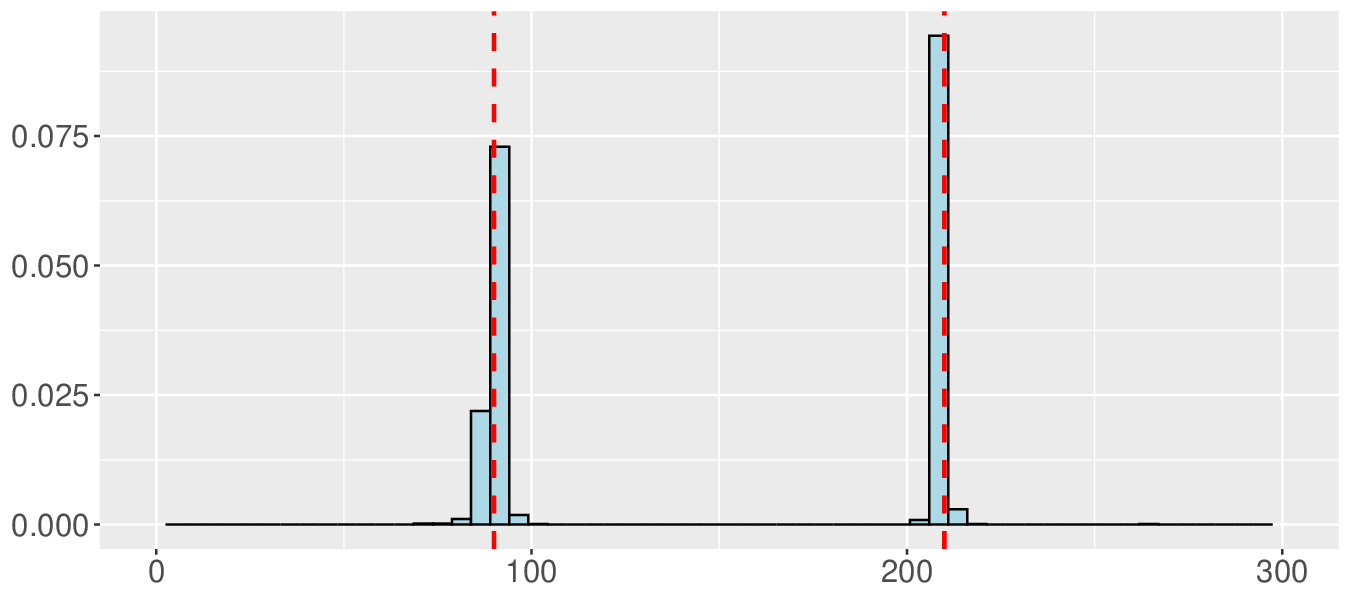}
	\caption{Frequency of $n_{(n,h)}^{max}$ with $N=300$. Top: $H_{A,1}^*$; Bottom: $H_{A,2}^*$.}
	\label{fig:cploc_dist}
\end{figure}

\section{Empirical  Applications}\label{sec:data}
In our empirical applications, we employ the proposed procedure of multiscale change point detection to the intraday curves of the Chicago Board Options Exchanged (CBOE) Volatility Index (VIX). The VIX index is a famous ``ready-to-use'' implied volatility, with real-time quote provided by the CBOE. It is a risk-neutral measure on the market's forward-looking expectation of 30-day (calendar day) volatility of the US stock market. It is derived by aggregating the weighted prices (mid-quote) of real-time of S\&P 500 index call and put options over a wide range of strike prices (out-of-money). The VIX index has gained substantial popularity in recent years, and it is now generally regarded as the primary benchmark for US stock market volatility, which is widely used by financial theorists, risk managers and volatility traders alike. The VIX index is often referred as the ``fear index'' in the financial media, and  its values above 30 indicate greater market fear and uncertainty, while values below 20 suggest stability.

Our dataset consists of high-frequency intraday observations of the VIX obtained from FirstRate Data\footnote{\url{https://firstratedata.com/}, assessed on 8 Nov 2025.}. With a focus on the recent decade, the sample period is chosen to be from January 4, 2016 to October 31, 2025.\footnote{VIX1Y starts from 2023-04-25, and VIX1D starts from 2018-04-26.} After excluding days with substantial missing data, there are 2,451 trading days in total. The CBOE disseminates VIX values during regular trading hours (9:30 a.m.–4:15 p.m. Eastern Time) and, since April 2016, also during global trading hours (3:15 a.m.–9:25 a.m. Eastern Time). We choose to focus on the regular trading hours, where the liquidity of the underlying S\&P 500 options is substantially higher, with a more reliable measurement of implied volatility.

The functional time series of intraday VIX curves are constructed from the high-frequency data between 9:30 a.m. and 4:15 p.m., with the time domain rescaled to $[0,1]$. Given the well-documented volatility clustering in financial markets, the error process is highly likely to be weakly dependent, consistent with the framework considered in Section \ref{sec:dep}. Consequently, we employ the long-run covariance estimator of \eqref{eq:WZ_est}. Then, this estimator is used within the bootstrap procedure (Algorithm \ref{alg2}) to determine the critical threshold at the 5\% significance level. Subsequently, Algorithm \ref{alg1} is applied with the obtained threshold to detect and localize multiple change points.

There are several choices of intraday sampling frequency. In our main analysis, we select a 5-minute sampling frequency, as this granularity provides a good balance between capturing intraday signals and mitigating the impact of market microstructure noises \cite{barndorff2002econometric}. To check the robustness of our findings, we provide a sensitivity analysis based on 1-minute and 30-minute frequencies. Beyond the standard 30-day horizon of the VIX, the CBOE employs the same methodology to construct similar volatility indices over alternative horizons, including VIX1D, VIX9D, VIX3M, VIX6M, and VIX1Y, corresponding to  forward-looking expectations over 1-day, 9-day, 3-month, 6-month, 1-year, respectively. To provide a comprehensive view of the market expectations across various horizons, we also report the results of detected changes for these related volatility indices.

\subsection{Main analysis}
Figure \ref{fig:vix5min3d} displays the intraday curves of the VIX index constructed from 5-minute frequency data. After applying the proposed multiscale procedure, we obtain the detected change points ($n_{(n,h)}^{max}$) shown in Table \ref{tab:VIX_5min_Changes}. The parameter $h$, which determines the confidence intervals ($[n^{\max}_{(n,h)}-h+1, n^{\max}_{(n,h)}+h]$), provides additional insight into the nature of each change. Smaller values of $h$ correspond to abrupt and short-lived breaks, such as those observed during the COVID-19 outbreak, whereas larger $h$ values indicate more gradual or persistent transitions, as seen during periods of monetary policy adjustment. The results show a strong correspondence between the detected change points and widely recognized episodes of volatility regime shifts in the US stock market. In the following, we highlight some prominent changes.

The initial set of detected changes during 2016--2018 aligns with major political and market events. There is a change point detected exactly on the date of the US Presidential Election (2016-11-08) that Donald Trump defeated Hillary Clinton to be elected 45th US President. This event triggered an overnight spike in VIX futures followed by a sharp intraday reversal during the regular session. The detection on 2018-02-01 (with $h=10$) results in a confidence interval [2018-01-19, 2018-02-15] that correctly envelops the ``Volmageddon'' event on 2018-02-05, when a VIX spike induced the liquidation of inverse-volatility exchange-traded products.

Since the outbreak of COVID-19, the procedure detects a sequence of change points: the beginning of the COVID-19 crisis (detected 2020-02-21), the peak of market fear (detected 2020-03-10) with a confidence interval [2020-03-04, 2020-03-19] that covers three significant crashes, and the shift to a new high-volatility regime following central bank interventions (detected 2020-04-02). The last detection in this period is on 2020-11-10, immediately following the Pfizer vaccine announcement (2020-11-09) that caused a fundamental change in volatility expectations.

The changes in 2022-2023 period are related to monetary policy adjustment. The change on 2022-03-16 coincides with the first U.S. Federal Reserve interest rate hike of the cycle.  Its confidence interval [2022-02-24, 2022-04-06] begins exactly on the date of the Russia-Ukraine conflict, which had introduced significant uncertainty to global energy markets. The detection on 2022-09-13 aligns with a major market sell-off following a high US CPI report. The detection on 2023-05-12  corresponds to the U.S. regional banking crisis, following the failures of Silicon Valley Bank and First Republic Bank.

The most recent detections correspond to significant geopolitical and trade-related announcements. A notable change point is detected on 2025-04-03, immediately following the ``Liberation Day'' tariff announcements on 2025-04-02. The subsequent detection on 2025-05-09 identifies a period of sustained uncertainty related to US-China trade negotiations.

\begin{figure}
	\centering
	\includegraphics[width=1\linewidth]{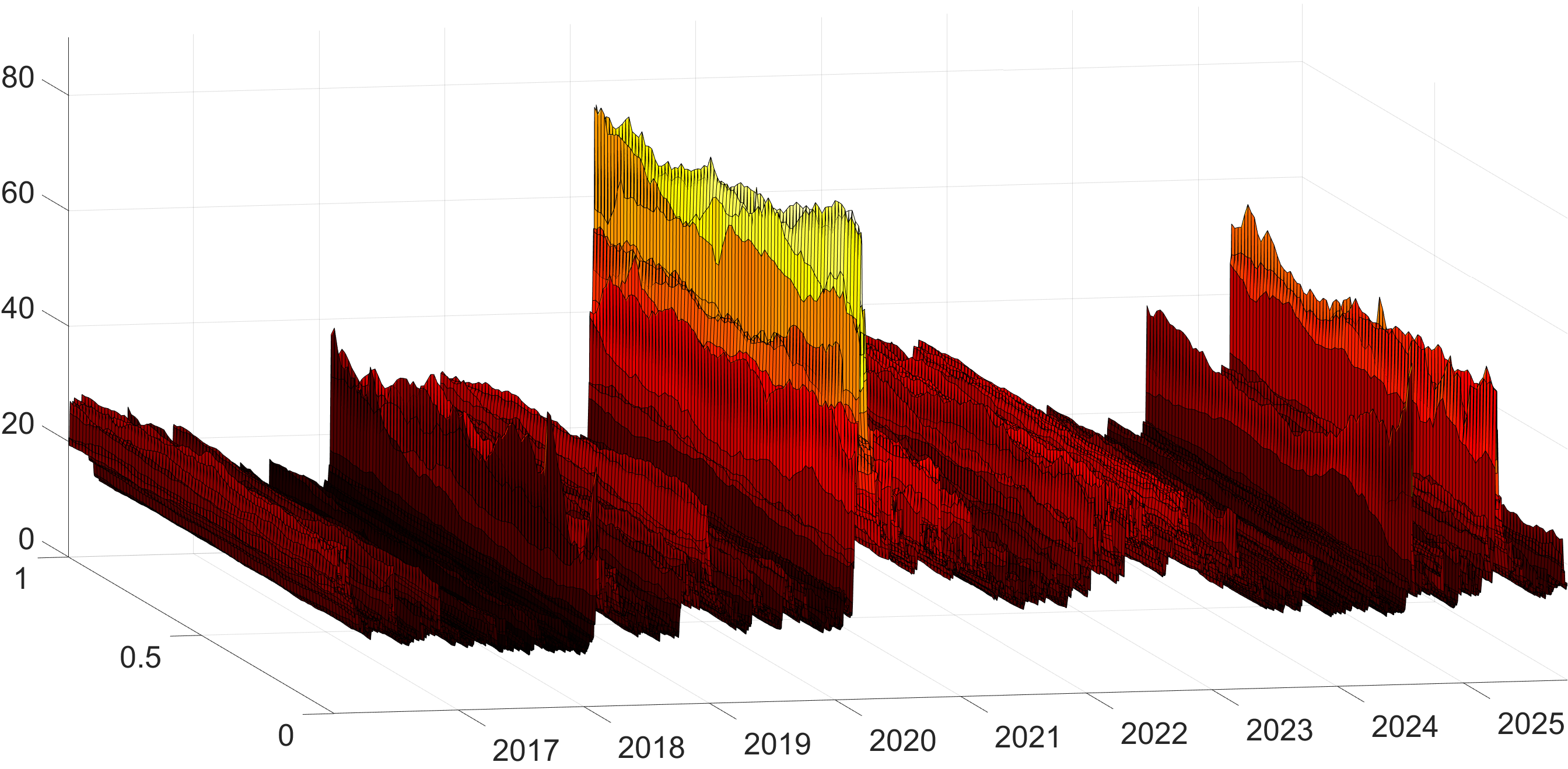}
	\caption{Intraday VIX Curves in 5-minute frequency}
	\label{fig:vix5min3d}
\end{figure}

% Table generated by Excel2LaTeX from sheet 'VIX_5min'
\begin{table}[htbp]
	\begin{center}
		\caption{Detected changes for the intraday VIX curves in 5-minute frequency\label{tab:VIX_5min_Changes}}
		\resizebox{\columnwidth}{!}{ \begin{tabular}{clcrlp{0.7cm}clcrl}
				\toprule
				\toprule
				$n_{(n,h)}^{max}$ & $h$   & $\gamma(n, h)$ & $[n^{\max}_{(n,h)}-h+1,$ &  $n^{\max}_{(n,h)}+h]$ &       & $n_{(n,h)}^{max}$ & $h$   & $\gamma(n, h)$ & $[n^{\max}_{(n,h)}-h+1,$ &  $n^{\max}_{(n,h)}+h]$ \\
				\cmidrule{1-5}\cmidrule{7-11}    2016-02-29 & 23    & 91.40 & [2016-01-27, & 2016-04-04] &       & 2020-11-10 & 17    & 92.49 & [2020-10-19, & 2020-12-07] \\
				2016-11-08 & 106   & 94.63 & [2016-06-10, & 2017-04-13] &       & 2021-03-12 & 34    & 94.27 & [2021-01-25, & 2021-04-30] \\
				2018-02-01 & 10    & 93.04 & [2018-01-19, & 2018-02-15] &       & 2022-01-20 & 23    & 92.90 & [2021-12-17, & 2022-02-23] \\
				2018-05-04 & 34    & 91.90 & [2018-03-19, & 2018-06-22] &       & 2022-03-16 & 15    & 93.49 & [2022-02-24, & 2022-04-06] \\
				2018-10-08 & 17    & 91.38 & [2018-09-14, & 2018-10-31] &       & 2022-09-13 & 23    & 91.13 & [2022-08-11, & 2022-10-14] \\
				2019-01-23 & 23    & 92.23 & [2018-12-18, & 2019-02-26] &       & 2022-12-27 & 49    & 95.88 & [2022-10-17, & 2023-03-09] \\
				2019-10-11 & 49    & 91.67 & [2019-08-05, & 2019-12-23] &       & 2023-05-12 & 45    & 92.22 & [2023-03-10, & 2023-07-20] \\
				2020-02-21 & 7     & 96.61 & [2020-02-12, & 2020-03-03] &       & 2023-11-13 & 37    & 91.89 & [2023-09-22, & 2024-01-09] \\
				2020-03-10 & 5     & 93.61 & [2020-03-04, & 2020-03-19] &       & 2024-07-17 & 28    & 91.59 & [2024-06-05, & 2024-08-26] \\
				2020-04-02 & 9     & 104.16 & [2020-03-23, & 2020-04-16] &       & 2025-04-03 & 6     & 102.49 & [2025-03-27, & 2025-04-11] \\
				2020-07-15 & 23    & 93.37 & [2020-06-12, & 2020-08-17] &       & 2025-05-09 & 19    & 95.77 & [2025-04-14, & 2025-06-06] \\
				\bottomrule
				\bottomrule
		\end{tabular}}%
	\end{center}
	{\footnotesize Note: The critical threshold based on the bootstrap procedure is 91.09 at the 5\% significance level.  }
\end{table}%

\subsection{Sensitivity analysis}
To assess the sensitivity of the proposed procedure to the choice of intraday sampling frequency, we repeat the analysis using the VIX intraday curves constructed from 1-minute and 30-minute data. Figure \ref{fig:res_3freq} displays the detected change points across the three frequencies (1-, 5-, and 30-minute). The results show a high degree of alignment, with all three frequencies yielding 22 detected changes. This consistency indicates that the proposed multiscale method is robust to the choice of intraday sampling resolution. Minor discrepancies in the precise timing of a few detections arise primarily during high volatility periods, such as early 2020, reflecting differences in the degree of noise smoothing at each sampling frequency. Overall, the general alignment of detected changes across frequencies confirms the robustness of the results on the change point detection in the intraday VIX curves.

\begin{figure}
	\centering
	\includegraphics[width=1\linewidth]{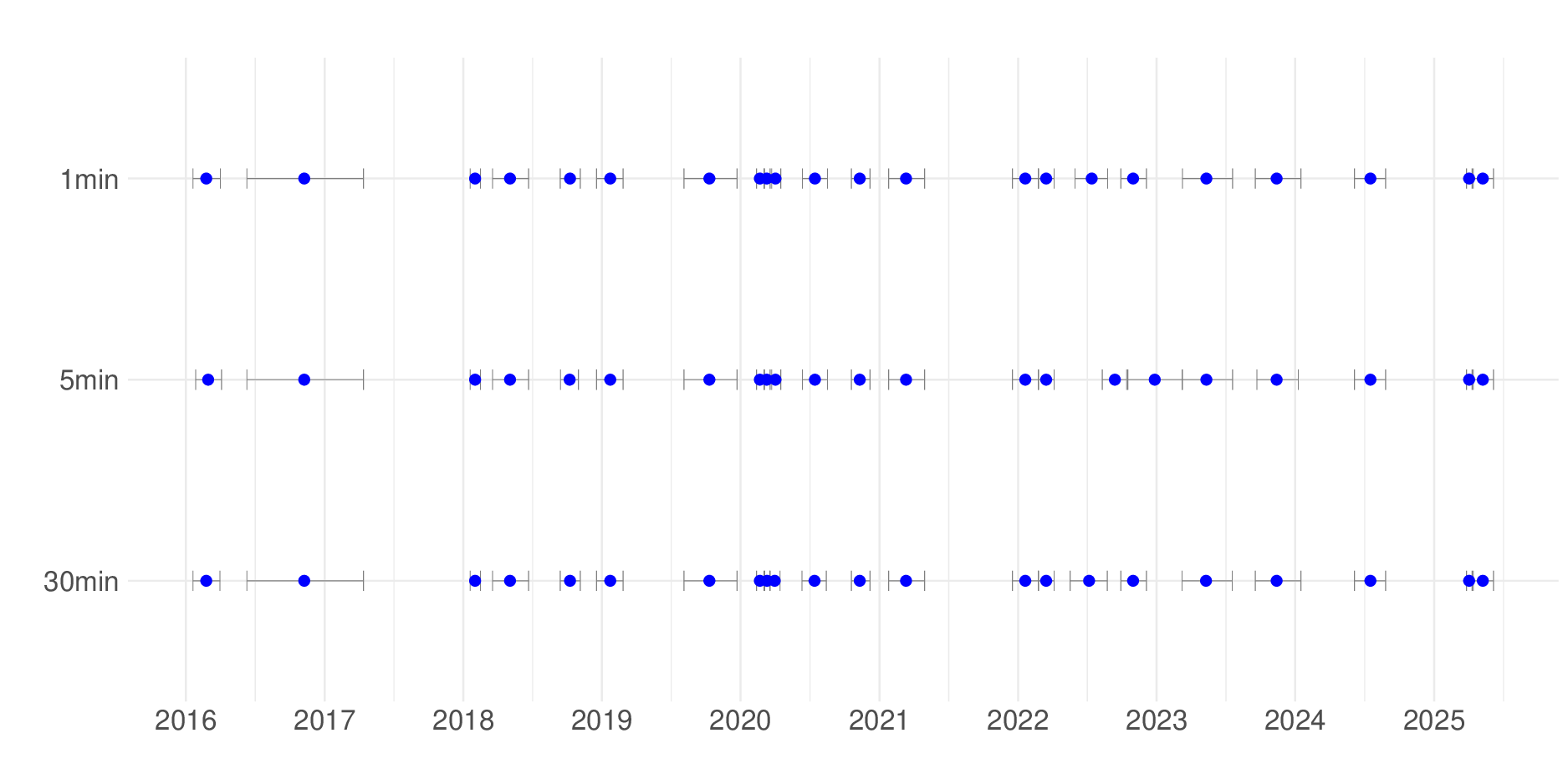}
	\caption{Detected changes for VIX of three intraday frequencies}
	\label{fig:res_3freq}
\end{figure}

\subsection{Volatility index of other horizons}
The standard VIX index reflects the market’s forward-looking expectation of 30-day volatility. To investigate change points for the volatility expectation at other horizons, we extend our analysis to a family of related volatility indices constructed using the same methodology but over alternative horizons. Specifically, we apply the proposed multiscale procedure to the intraday volatility curves (based on 5-minute data) of  VIX1D (1-day), VIX9D (9-day), VIX3M (3-month), VIX6M (6-month), and VIX1Y (1-year). Figure \ref{fig:res_6horizon} displays the detected change points across these indices. The results show that the timing of change points is broadly synchronized across horizons, suggesting that volatility changes tend to occur simultaneously in the full spectrum on the volatility term structure, including short-term, medium-term, and long-term horizons.

\begin{figure}
	\centering
	\includegraphics[width=1\linewidth]{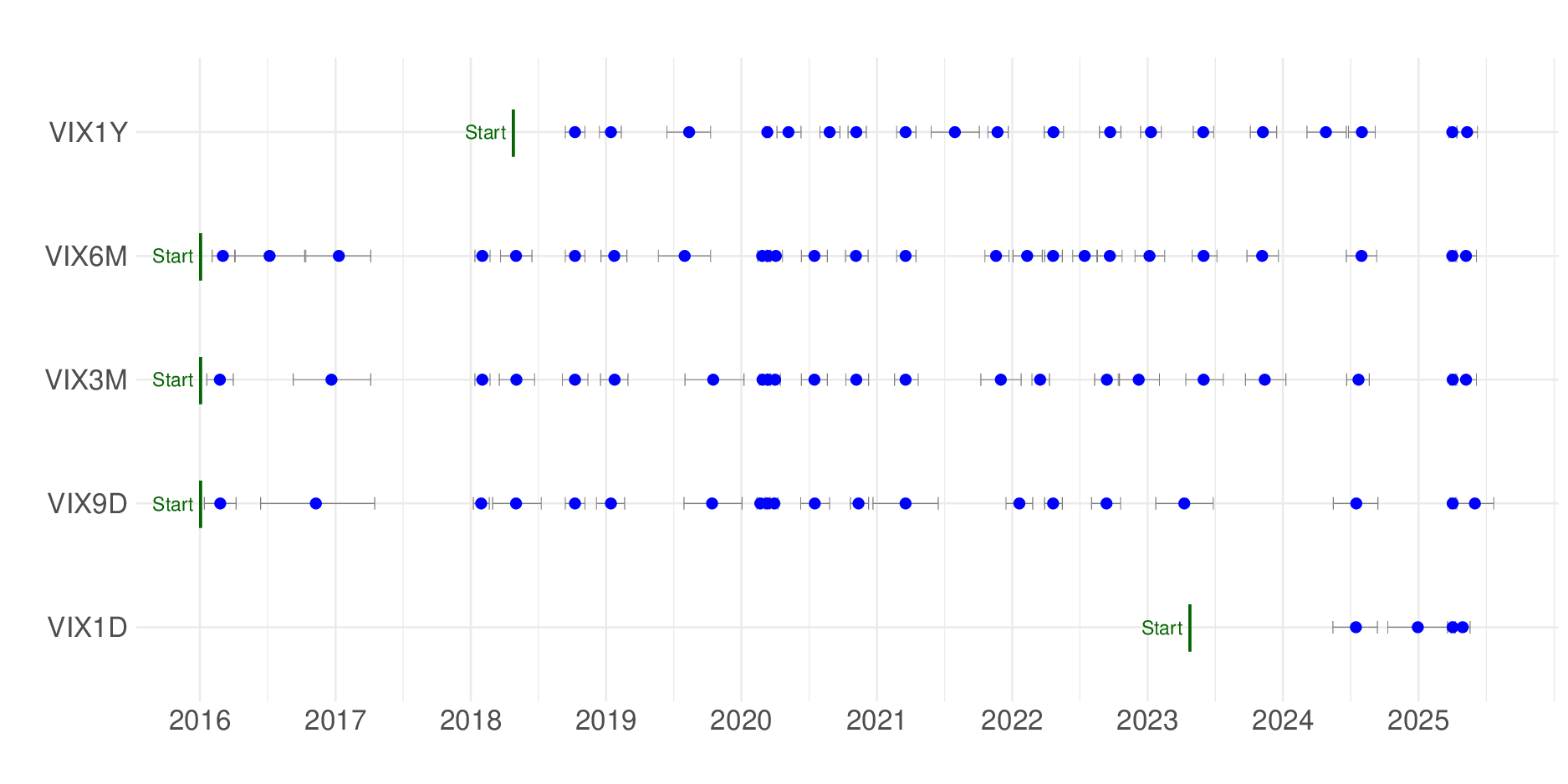}
	\caption{Detected changes for volatility indices of six horizons}
	\label{fig:res_6horizon}
\end{figure}

\bigskip

{\bf Acknowledgments. }
This work was supported by TRR 391 \textit{Spatio-temporal Statistics for the Transition of Energy and Transport} (Project number 520388526) funded by the Deutsche Forschungsgemeinschaft (DFG, German Research Foundation). Tim Kutta's work has been partially funded by AUFF grants 47331 and 47222.

\small\renewcommand{\baselinestretch}{0.95}
\putbib

%\bibliographystyle{chicago}
%{\small\renewcommand{\baselinestretch}{0.95}
%\bibliography{Literature}
%}
 $ $ \newpage
\end{bibunit}
\begin{bibunit}
\setcounter{page}{1}
%\pagenumbering{arabic}
\section{Proofs and technical details}

The appendix is dedicated to the proofs of our main results. Throughout this section we use a notation to compare the size of two sequences of positive numbers $(a_N)_{N \in \mathbb{N}}$ and $(b_N)_{N \in \mathbb{N}}$. We write $a_N \lesssim b_N$, if there exists a non-random constant $C>0$, independent of $N$ such that $a_N \leq C\, b_N$ for all $N\ge 1$.

\subsection{Proof of Theorem \ref{theo1}} \label{sec:proofs1}

\textbf{Preliminary approximations } According to Assumption \ref{ass_1}, we know that on the Hölder space $\mathcal{C}^\rho([0,1])$ the weak convergence $P_N \to W$ holds, where $W$ is a centered Gaussian process. Notice that the Hölder space itself is again a Banach space and also separable \citep[see][]{rackauskas:suquet:2009}.
Using Skorohod's representation theorem (see \cite{billingsley:1999}, Theorem 6.7), we can move to a probability space where the convergence $P_N \to W$ holds a.s. and we fix the outcome so that $P_N \to W$ holds deterministically 
in the Hölder space.  Next, we define the mean functions
\begin{align} \label{e:def:f}
f(n,h)= \sum_{i=n-h+1}^n \mathbb{E}X_i - \sum_{i=n+1}^{n+h} \mathbb{E}X_i
\end{align}
and notice that we can therewith rewrite the scan statistic in \eqref{e:gam} as 
\begin{align}\label{e:def:gam}
\gamma(n,h) = \frac{\|[P_N(n/N)-P_N((n-h)/N)]-[P_N((n+h)/N)-P_N(n/N)]+f(n,h)/\sqrt{N}\|}{ \rho(h/N)}.
\end{align}
We also define the Gaussian version of the function $\gamma$ as
\begin{align*}
\tilde \gamma(n,h) := \frac{\|[W(n/N)-W((n-h)/N)]-[W((n+h)/N)-W(n/N)]+f(n,h)/\sqrt{N}\|}{ \rho(h/N)}.
\end{align*}
%where $W$ is the limiting Gaussian process in Assumption \ref{ass_1}.
If there is no change in the interval characterized by $(n,h)$, it follows that $\gamma(n,h)=\gamma_0(n,h)$ with 
\begin{align}\label{e:def:gam0}
\gamma_0(n,h):=\frac{\|[P_N(n/N)-P_N((n-h)/N)]-[P_N((n+h)/N)-P_N(n/N)]\|}{ \rho(h/N)}.
\end{align}
We define $\tilde \gamma_0$ analogously, but with $P_N$ replaced by $W$. 
Using the triangle inequality for the  norm of the Banach space twice (first the reverse, then the usual version) it follows that
\begin{align}\label{e:maxgg0}
\max_{(n,h) \in \mathbf{N}_N} |\tilde \gamma_0(n,h)-\gamma_0(n,h)| \le
& \max_{(n,h) \in \mathbf{N}^{all}} |\tilde \gamma_0(n,h)-\gamma_0(n,h)| \\
\le & 2 \max_{(n,h) \in \mathbf{N}^{all}} \frac{\|[W-P_N](n/N)-[W-P_N]((n-h)/N)\|}{\rho(h/N)}\nonumber\\
\le & 2 \sup_{0 \le x <y \le 1}  \frac{\|[W-P_N](y)-[W-P_N](x)\|}{\rho(y-x)}
\le 2 \|W-P_N\|_\rho =o(1). \nonumber
\end{align}
In the second to last step, we have used the definition of the Hölder norm, see eq. \eqref{e:Hnorm}. In the last step, we have used convergence  of $P_N$ to $W$ in the Hölder space.

\textbf{Proof of Theorem \ref{theo1}} The algorithm MultiScan$(\emptyset, \mathbf{N}_N, q_{1-\alpha} )$ can only include a pair $(n,h) $ in its output $\mathbf{C}$, if $\gamma(n,h)>q_{1-\alpha} $. If $(n,h)$ is falsely included, this means that there is no change in $\{ n-h+1,\ldots ,n+h\}$ and hence, we have $\gamma(n,h)=\gamma_0(n,h) $.
Obviously we have the implication
\begin{align} \label{e:Sinq}
\exists (n,h) \in \mathbf{N}_N \,\, \textnormal{falsely included in} \,\, \mathbf{C} \quad\Rightarrow \quad\max_{(n,h) \in \mathbf{N}_N}\gamma_0(n,h)>q_{1-\alpha} 
.
\end{align}
By \eqref{e:maxgg0} we have
\[
\Big |\max_{(n,h) \in \mathbf{N}_N}\gamma_0(n,h)-\max_{(n,h) \in \mathbf{N}_N}\tilde \gamma_0(n,h) \Big |=o(1).
\]
Thus we can focus on the convergence of $\max_{(n,h) \in \mathbf{N}_N}\tilde \gamma_0(n,h)$ in the following discussion. By Lemma \ref{lem:conv}  (which is stated below) it follows that 
\begin{align} \label{e:Lrcon}
\max_{(n,h) \in \mathbf{N}_N}\tilde \gamma_0(n,h) \to L^\rho ,
\end{align}
where  $L^\rho$ is defined in \eqref{e:defLrho}. The above almost sure convergence implies, of course, convergence in distribution.
Now, since $L^\rho$  has a continuous distribution function (see Lemma \ref{lem:cont} below), it follows that
\begin{align} \label{e:Sineq}
\begin{split}
\mathbb{P}\big(\exists (n,h) \in \mathbf{N}_N \,\, \textnormal{falsely included in} \,\, \mathbf{C} \big) & \le  \mathbb{P}\big(\max_{(n,h) \in \mathbf{N}_N}\gamma_0(n,h)>q_{1-\alpha} \big) \\
& \to \mathbb{P}(L^\rho > q_{1-\alpha})=\alpha.
\end{split}
\end{align}
In the case where no changes occur, it is is clear that the implication in \eqref{e:Sinq} becomes an equivalence and hence, the inequality sign in \eqref{e:Sineq} becomes an equality sign.   The proof will be completed by stating and proving  
Lemmas \ref{lem:conv} and \ref{lem:cont}.
\qed
$ $\\
\begin{lem} \label{lem:conv}
     Under the assumptions of Theorem \ref{theo1}, the convergence  in \eqref{e:Lrcon} follows.
\end{lem}

\begin{proof}[Proof of Lemma \ref{lem:conv}] For the proof, the reader should recall the definition of the triangle $\blacktriangle_{0,1}$ in eq. \eqref{e:def:tri} and the set $\mathbf{N}_N$ as one of the sets in eq. \eqref{e:setchoice}. We also point out to the reader that the set inclusion $\mathbf{N}_N/N \subset \blacktriangle_{0,1}$ holds (this is obvious from the construction of the respective sets). We now turn to the proof.\\
    The convergence holds for any fixed outcome where the sample path of $W$ lies in $\mathcal{C}_\rho([0,1])$.  As before, we fix  such an outcome and notice that by definition of $L^\rho$ in \eqref{e:defLrho} we have that 
     \begin{align}
     \nonumber 
    \max_{(n,h) \in \mathbf{N}_N}\tilde \gamma_0(n,h)  & \le \max_{(x,y) \in \mathbf{N}_N/N} f^W(x,y)
    \\
    &  \leq \sup_{(x,y) \in \blacktriangle_{0,1}} f^W(x,y) = L^\rho , 
    \label{det3}
     \end{align}
     where the function $f^W$ is defined by 
     \[
f^W(x,y):=\frac{\|W(x+y)-2W(x)+W(x-y)\|}{\rho(y)}.
\]  
Since $W$ is Hölder continuous it follows directly that $f^W$ is continuous on $\blacktriangle_{0,1}$. 
Again by construction $\mathbf{N}_N/N$ (see eq. \eqref{e:setchoice}) is asymptotically dense in $\blacktriangle_{0,1}$ in the sense that for any $(x,y) \in \blacktriangle_{0,1}$ there exists a sequence $(x_N,y_N) \in \mathbf{N}_N/N$ such that $x_N \to x, y_N \to y$. This holds in particular for a point in (the compact) set $\blacktriangle_{0,1}$ where $f^W$ is maximized, say $(x^*, y^*)$,  and it follows that
\[
\max_{(n,h) \in \mathbf{N}_N}\tilde \gamma_0(n,h) = \max_{(x,y) \in \mathbf{N}_N/N} f^W(x,y)\ge f^W(x_N^*, y_N^*) \to f^W(x^*,y^*)=L^\rho.
\]
Observing the upper bound in \eqref{det3}  the the convergence  in \eqref{e:Lrcon} follows.
\end{proof}

\begin{lem} \label{lem:cont}
    Under Assumption \ref{ass_1} and for any of the weight functions in \eqref{polyweights} and \eqref{logweights}, the random variable $L^\rho$ has a continuous distribution function.
\end{lem}

\begin{proof}[Proof of Lemma \ref{lem:cont}]  For the proof, we first notice that $L^\rho$, defined in \eqref{e:defLrho}, equals the supremum norm of the continuous (random) function 
\[
g^W(x,y):=\frac{W(x+y)-2W(x)+W(x-y)}{\rho(y)}, \qquad (x,y) \in \blacktriangle_{0,1}.
\]
For any choice $x,y$ the element $g^W(x,y)$ lives in the Banach space $\mathbb{B}$. Due to continuity of $g^W$, we can confine the supremum in the definition of $L^\rho$  in \eqref{e:defLrho} to the countable set $(x,y) \in \blacktriangle_{0,1} \cap \mathbb{Q}^2$. Next, since $\mathbb{B}$ is separable, we choose a dense, countable subset $\mathcal{B}$ from it. For every element $b \in \mathcal{B}$, we can find a functional $b'$ from the topological dual space with $\|b'\|_\mathcal{L}= 1$ and $b'(b)=\|b\|$  \citep[this is a well-known consequence of the Hahn-Banach theorem; see Theorem 12.2 in][]{bachman:narici:2000}. Then, it follows for any $\bar b \in \mathbb{B}$ that
\begin{align} \label{e:supid}
\sup_{b \in \mathcal{B}}b'(\bar b)=\|\bar b\|.
\end{align}
The inequality "$\le$" is clear because $b'(\bar b) \le \|b'\|_\mathcal{L} \|\bar b\| = \|\bar b\|$.
For the reverse direction let $\epsilon>0$ be fixed, but arbitrary and choose an element $b_\epsilon \in \mathcal{B}$ such that $\|b_\epsilon-\bar b\|\le \epsilon$. Then, it follows that
\[
b_\epsilon'(\bar b) = b_\epsilon'(b_\epsilon)+b_\epsilon'(\bar b-b_\epsilon) =  \|b_\epsilon\| +b'_\epsilon(\bar b-b_\epsilon) \ge \|\bar b\| -\epsilon-|b_\epsilon'(\bar b-b_\epsilon)|.
\]
Now, since $\|b_\epsilon'\|_\mathcal{L}= 1$ we have 
\[
|b_\epsilon'(\bar b-b_\epsilon)| \le \|b_\epsilon'\|_\mathcal{L} \|\bar b-b_\epsilon\| \le \epsilon ,
\]
and we obtain 
\[
b_\epsilon'(\bar b) \ge \|\bar b\|-2\epsilon.
\]
Since $\epsilon>0$ was arbitrary, we obtain the identity \eqref{e:supid}. Accordingly, we have 
\[
L^\rho =
\sup_{(x,y) \in \blacktriangle_{0,1} \cap \mathbb{Q}^2} \| g^W(x,y)\| = 
\sup_{(x,y) \in \blacktriangle_{0,1} \cap \mathbb{Q}^2} \sup_{b \in \mathcal{B}}\,\, \big[b'(g^W(x,y))\big].
\]
Notice that $b'(g^W(x,y))$ is for each $x,y,b$ a real-valued normal distributed random variable, i.e. we are taking the supremum of the  separable Gaussian process 
$\{ b'(g^W(x,y)) \colon  b \in \mathcal{B}; (x,y) \in \blacktriangle_{0,1} \} $
that is almost surely finite and non-degenerate (the latter, because the covariance is assumed to be non-trivial in Assumption \ref{ass_1}). It is a direct consequence of Theorem 2 in \cite{giessing:2023} that the cdf of $L^\rho$ is  continuous. 
\end{proof}

\subsection{Proof of Theorem \ref{theo2}}

For the proof, it is helpful for the reader to recall the objects $\gamma, \gamma_0$ and $\tilde \gamma_0$, all defined at the beginning of Section \ref{sec:proofs1}. We also recall their relation briefly. $\gamma$ is the central test statistic that our Algorithm \ref{alg1} uses for change point detection. If $\gamma(n,h)>q_{1-\alpha}$ the algorithm detects a change in the interval $[n-h+1, n+h]$ characterized by $n,h$. Let us now consider the case, where there actually is no change in said interval. In this case it holds by definition that $\gamma(n,h)=\gamma_0(n,h)$, since $\gamma_0(n,h)$ is the test statistic with all means removed. We will thus consider repeatedly the event $\gamma_0(n,h)>q_{1-\alpha}$, which implies that an erroneous  detection of a change took place (see also eq. \eqref{e:Sinq}). Finally, $\tilde \gamma_0(n,h)$ is a Gaussian approximation of $\gamma_0(n,h)$ that is more manageable than the original $\gamma_0(n,h)$.
 
\textbf{Outline } To explain  the structure of the proof given below, we give a short intuition first. To show strong localization, we need two things: First, and obviously, the statistic $\gamma(n,h)$ needs to be sufficiently big ($\gamma(n,h)>q_{1-\alpha}$) if $n$ is a change, or near a change. This is the subject of step II of the proof  below. Next, and perhaps less obviously, detecting a true change could fail in Algorithm \ref{alg1}, even if $\gamma(n,h)$ is  large, because the algorithm occasionally detects  a spurious change, say in $(n_{spur}, h_{spur})$. In Theorem \ref{theo1} provides  weak localization which means that the probability of this event is asymptotically bounded by $\alpha$, but it is typically not $0$. If a change is wrongly detected in   $(n_{spur}, h_{spur})$ the algorithm discards values with overlapping intervals 
and this just might delete the pair $(n,h)$ and thus prevent us from detecting a true change later on. Step I of the proof lays the foundations for our argument why this (asymptotically) never happens. The reasoning works like this:  Algorithm \ref{alg1} goes through all pairs $(n,h)$ from small $h$ to large $h$. However, if there is no change in $[n-h+1, n+h]$ we have $\gamma(n,h)=\gamma_0(n,h)$ and we show in step I that for all small $h$ we have  $ \gamma_0(n,h) \approx 0$. More precisely, we show
\[
\max_{1 \le h \le N u_N} \gamma_0(n,h) \to 0
\]
whenever  $u_N$ is an (arbitrary) null sequence. In other words, spurious rejections only occur for large $h$ with $h/N$ not negligible and thus, they occur very late when running the algorithm. In particular, under the assumptions of the theorem, we will prove  that all true changes are detected for $h\ll N u_N$ and thus before the spurious changes can hamper their detection. Step III pulls together the different threads of the argument. Finally, since the main body of the proof is formulated for polynomial weight functions, we provide some details on the proof for logarithmic weight functions in the end.

\textbf{Step I }
We use the same notations as in the previous sections and as before assume that the probability space has been modified such that $P_N \to W$ a.s.
To start, we notice the following aspect of the quantity  $\gamma_0(n,h)$ defined in \eqref{e:def:gam0}, which coincides with scan statistic $\gamma(n,h)$, if there is no change 
in the interval characterized by $(n,h)$: Let $u_N$ be a sequence of positive numbers with $u_N \downarrow 0$, 
then a.s. 
\begin{align} \label{e:nospur}
\max_{(n,h) \in \mathbf{N}_N: h <Nu_N}\gamma_0(n,h) \to 0, 
\end{align}
because it follows from \eqref{e:maxgg0}
 that a.s. 
\begin{align} \label{e:g:s1}
\max_{(n,h)}|\gamma_0(n,h)-\tilde \gamma_0(n,h)|=o(1)
\end{align}
where $\tilde \gamma_0$ is defined 
as $\gamma_0$ in \eqref{e:def:gam0}, but with $P_N$ replaced by $W$.
Now, using this definition
and similar arguments  as in the derivation of  \eqref{e:maxgg0}
for the first inequality, we have 
\begin{align}\label{e:g:s2}
    \max_{(n,h): h \le N u_N}|\tilde \gamma_0(n,h)| \le 2 \sup_{\substack{0 \le x < x+y \le 1 \\ 0 < y \le u_N}} \frac{\|W(x)-W(x+y)\|}{\rho(y)}=o(1),
\end{align}
where  the second estimate follows form the fact  that $W$ is   an element of the Hölder space.
This statement is true for any choice of $(u_N)_N$. 

We will now fix a  sequence $(u_N)_N$ that satisfies
\begin{align} \label{e:prop:uN}
  \frac{\max_k r_{N,k}}{N u_N}\to 0, \,\,\, \textnormal{where} \,\,\, r_{N,k} :=  
  \begin{cases}
    \bigg(\frac{N^{1/2-\beta}}{\Delta_k}\bigg)^{\frac{1}{1-\beta}}, & \textnormal{for}\,\, \rho(x)=x^\beta \\
    \frac{\log^{2\beta}(N)}{\Delta_k^2}, & \textnormal{for}\,\, \rho(x)=x^{1/2}\log^\beta(x^{-1}). \\
\end{cases}
\end{align}
The above condition can always be fulfilled if $\max_ k r_{N,k} = o(N)$ for a suitable choice of $(u_N)_N$. We verify that indeed $\max r_{N,k} = o(N)$ with a small calculation in the case of polynomial weights (the case of logarithmic weights  can be verified similarly). Recall Condition i) in Assumption \ref{ass_2}, which states that $\min_{k=1,...,K} \delta_k^{1-\beta}\Delta_k\gg N^{1/2-\beta}$. Now
\begin{align*}
     \max_k r_{N,k} & = \max_k \delta_k \bigg(\frac{N^{1/2-\beta}}{\delta_k^{1-\beta}\Delta_k}\bigg)^{\frac{1}{1-\beta}} \le  \max_k (\delta_k ) \bigg(\frac{N^{1/2-\beta}}{\min_k (\delta_k^{1-\beta}\Delta_k)}\bigg)^{\frac{1}{1-\beta}} \\
   & =  \max_k (\delta_k) \cdot o(1) \le N  \cdot o(1)=o(N).
\end{align*}
Thus, we can choose a sequence $(u_N)_N$ that converges  slowly enough to $0$ such that the convergence in \eqref{e:prop:uN} holds.
Now, observing  \eqref{e:g:s1} and  \eqref{e:g:s2}), we have 
\[
\mathbb{P} \big (  \mathcal{E}_N  \big )  \to 0,
\]
where
$$
\mathcal{E}_N:=  \Big \{   \max_{(n,h): h \le N u_N} \gamma_0(n,h)>q_{1-\alpha}\Big \}.
$$
Thus, asymptotically, there will be \textit{no pairs wrongly} included in the output $(n,h) \in \mathbf{C}$ of \\ MultiScan$(\emptyset, \mathbf{N}_N, q_{1-\alpha})$  where $h\le N u_N$ (see also eq. \eqref{e:Sinq}).

\textbf{Step II } We now condition on the event $\mathcal{E}_N^c$ and analyze $\gamma(n,h)$. For the below decomposition, it is helpful to review the definitions of $\gamma$ (eq. \eqref{e:def:gam}), $\gamma_0$ (eq. \eqref{e:def:gam0}) and the function $f$ (eq. \eqref{e:def:f}). 
We then get, for any  $h\le u_N N$, the lower bound
\[
\gamma(n,h) = \frac{\|S_{n-h+1}^n-S_{n+1}^{n+h}\|}{\rho(h/N)N^{1/2} } \ge \frac{\|f(n,h)\|}{\rho(h/N)N^{1/2}}- \gamma_0(n,h) \ge \frac{\|f(n,h)\|}{\rho(h/N)N^{1/2}}- q_{1-\alpha},
\]
where the last inequality holds because we have conditioned on the set  $\mathcal{E}_N^c$.
Let us now consider for each  location  $c_k$ a  bandwidth parameter $H_k<\delta_k/2$ that also satisfies $H_k\le N u_N$.
We obtain for the maximum
\begin{align} \label{def:Mk}
M_k & := \max_{h\le H_k, |n-c_k|<h}\gamma(n,h) \ge \gamma(c_k, H_k) \\
&  \ge  \frac{\|f(c_k, H_k)\|}{\rho(H_k/N)N^{1/2}}- q_{1-\alpha}  =
       \frac{H_k\Delta_k}{\rho(H_k/N)N^{1/2}} - q_{1-\alpha} , 
       \nonumber 
\end{align}
where we used the definition of $f$  in \eqref{e:def:f} in the last step. 

For the next step, we have to consider the case where $\rho$ is the polynomial weight and the logarithmic weight separately and we focus for now on the polynomial case. We give additional details on the logarithmic case at the end of this section (it largely works analogously). For now,  recall  the definition of $r_{n,k}$ in \eqref{e:prop:uN}. For the polynomial case, define \footnote{If $\mathbf{N}_N = \mathbf{N}^{all}$ it is directly clear that $(c_k, H_k) \in \mathbf{N}_N$ for $H_k= \lfloor (5 q_{1-\alpha})^{1/(1-\beta)}  r_{N,k}\rfloor +1$. For the case where $\mathbf{N}_N = \mathbf{N}_N^{\theta}$ we should choose the next natural number for $H_k$ larger than $\lfloor (5 q_{1-\alpha})^{1/(1-\beta)}  r_{N,k}\rfloor +1$ that satisfies again  $(c_k, H_k) \in \mathbf{N}_N$. We will ignore this minor technicality in the following.} 
\begin{align} \label{e:def:Hk}
H_k := \lfloor (5 q_{1-\alpha})^{1/(1-\beta)}  r_{N,k}\rfloor +1. 
\end{align}
 Notice that such a choice is possible while still maintaining $H_k < \delta_k/2$, since 
\[
\frac{H_k}{\delta_k/2} \lesssim \frac{r_{N,k}}{\delta_k}  = \bigg(\frac{  N^{1/2-\beta}}{\delta_k^{1-\beta}\Delta_k } \bigg)^{\frac{1}{1-\beta}} \le \bigg(\frac{  N^{1/2-\beta}}{\min_k \delta_k^{1-\beta}\Delta_k } \bigg)^{\frac{1}{1-\beta}} \to 0.
\]
Here the constant implied by "$\lesssim$" does not depend on $k$, and we have used the assumption that $\min_{k=1,...,K} \delta_k^{1-\beta}\Delta_k\gg N^{1/2-\beta}$. 
With this choice of $H_k$, we obtain
\begin{align*}
     \frac{H_k\Delta_k}{\rho(H_k/N)N^{1/2}}=\frac{H_k^{1-\beta}\Delta_k}{N^{1/2-\beta}}\ge  \frac{ 5q_{1-\alpha}}{2}.
\end{align*}
The right side is greater than $2q_{1-\alpha}$, and we obtain from \eqref{def:Mk} that 
\begin{align} \label{e:Mk2}
\min_k M_k >q_{1-\alpha}.%$$,
\end{align}
The interpretation of this result is that for all changes $c_k$ (simultaneously) as $N \to \infty$ there exists at least the pair $(c_k, H_k) \in \mathbf{N}_N$, such that $\gamma(c_k, H_k) >q_{1-\alpha}$, with probability converging  to $1$ and $H_k \le u_NN$. 
In view of the convergence \eqref{e:nospur}, the condition $\gamma(n, h) >q_{1-\alpha}$ can (asymptotically) only be satisfied for $h \le u_NN$ if the interval defined  by $(n,h)$ captures one of the change point locations $c_k$. So no spurious changes will be detected for $h \le u_NN$. We continue, as before, conditioning on the event $\mathcal{E}_N^c$.

\textbf{Step III } For the next and final step, the reader should again familiarize herself with the Algorithm \ref{alg1} - in particular with the notations of the ordering "$\preceq$" (see eq. \eqref{e:def:order}), with the definition of the elimination method "$\circleddash$" (see eq. \eqref{e:def:minus}) and finally, with the fact that the algorithm proceeds along the ordering $ \preceq$ from smallest to largest element.
The inequality \eqref{e:Mk2} means that for each $k$, there exists a pair 
\[
(n_k^*, h_k^*) \in 
\mathbf{N}^{(H_k)} := \{(n,h) \in \mathbf{N}_N: h\le H_k, |n-c_k|<h\}
\]
where $\gamma(n_k^*,h_k^*) >q_{1-\alpha}$ and $(n_k^*, h_k^*)$ is the unique element of $\mathbf{N}^{(H_k)}$ that the algorithm would include in $\mathbf{C}$ (if it is not deleted before). As soon as a pair $(n_k^*,h_k^*)$ is processed, the algorithm updates 
\[
\mathbf{N}_N \to [\mathbf{N}_N\circleddash (n_k^*,h_k^*)]
\]
which does not include any elements of $\mathbf{N}^{(H_k)}$ anymore (the change point $c_k$ will only be detected once).
Since $h_k^*\le H_k < \delta_k/2$ it follows that the intervals defined  by the pairs $(n_k^*, h_k^*)$ are non-overlapping for different $k$. Now, consider a pair $k\neq k'$ with $(n_k^*, h_k^*) \preceq (n_{k'}^*, h_{k'}^*)$. The change at position $k$ will be detected first by our algorithm. Since $(n_k^*, h_k^*) \preceq (n_{k'}^*, h_{k'}^*)$ and the implied intervals are non overlapping, it follows that still 
\[
(n_{k'}^*, h_{k'}^*) \in [\mathbf{N}\circleddash (n_k^*,h_k^*)]
\]
for any subset $\mathbf{N} \subset \mathbf{N}_N$ with $(n_{k'}^*, h_{k'}^*), (n_k^*,h_k^*) \in \mathbf{N}$.
By induction it follows that the deletion of the first $r$ pairs (w.r.t. $\preceq$) does not delete the remaining pairs and hence all changes will be detected with their specific pair $(n_k^*,h_k^*)$. So, all $K$ changes are detected. This shows the strong localization property.

\textbf{Details on the logarithmic weight function } Finally, we have focused in our above proof on the case of a polynomial weight function. The case of logarithmic weights is similar. In this case, we define 
$H_k :=\lfloor  (5q_{1-\alpha})^2 r_{N,k} \rfloor+1$ 
which leads to the analysis of the deterministic part 
\begin{align} \label{e:logcalc}
\frac{H_k\Delta_k}{\rho(H_k/N)N^{1/2}} = \frac{H_k^{1/2}\Delta_k}{\log^\beta(N/H_k)}\ge \frac{5q_{1-\alpha}}{2}>2q_{1-\alpha}.
\end{align}
The remaining arguments are the same as in the polynomial case and therefore not repeated.

\subsection{Proof of Propositions \ref{prop1} and \ref{prop2}}

For this proof, it will be useful, to recall the definition of the limiting distribution $L^\rho$ in \eqref{e:defLrho}. Note  that the ratio on the right of  this equation is defined as $0$, if $y=0$;  this matters below for the identity \eqref{e:LLz}. As before, the reader should appreciate that $\mathbf{N}_N/N \subset \blacktriangle_{0,1}$. The reader should also be familiar with the definitions at the beginning of Section \ref{sec:proofs1} of $\gamma, \gamma_0$ and $\tilde \gamma_0$.
\\
We  mainly focus on the weak localization property in Proposition  \ref{prop1} and  recall that for any pair $(n,h)$ where no change occurs in the corresponding interval, we have
\begin{align*}
\gamma(n,h) &= \gamma_0(n,h)  \\
&=  \frac{\|[P_N(n/N)-P_N((n-h+1)/N)]-[P_N((n+h)/N)-P_N((n+1)/N)]\|}{ \rho(h/N)}.
\end{align*}
Since, by assumption of the theorem, the variables $\varepsilon_n$ are actually the increments of a Gaussian process $W$, we have that $P_N(n/N)=W(n/N)$ for any $n=0,...,N$ and thus
\begin{align*}
\gamma_0(n,h) & =\tilde \gamma_0(n,h) \\
& =  \frac{\|[W(n/N)-W((n-h+1)/N)]-[W((n+h)/N)-W((n+1)/N)]\|}{ \rho(h/N)}.
\end{align*}
Now, the random variable $L^\rho$ in this case is a.s. finite, since eq. \eqref{e:WCadd} implies that $\|W\|_\rho<\infty$ almost surely, and we have from \eqref{det3}
\[
\max_{(n,h) \in \mathbf{N}_N} \tilde \gamma_0(n,h) \le L^\rho. 
\]
In particular, if $q_{1-\alpha}$ denotes the $1-\alpha$-quantile of $L^\rho$, it follows as in the proof of Theorem \ref{theo1} (see eq. \eqref{e:Sinq}), that
\begin{align*}
\mathbb{P}\big(\exists (n,h) \in \mathbf{N}_N \,\, \textnormal{falsely included in} \,\, \mathbf{C} \big) &  \le  \mathbb{P}\big(\max_{(n,h) \in \mathbf{N}_N}\tilde \gamma_0(n,h)>q_{1-\alpha} \big) \\
& \le \mathbb{P}(L^\rho>q_{1-\alpha}) =\alpha.
\end{align*}
This is already the weak localization property, as defined in eq. \eqref{e:weakl}. Next, we show that if no change occurs the level $1-\alpha$ is asymptotically exactly approximated, i.e.
\begin{align*}
\mathbb{P}\big(\exists (n,h) \in \mathbf{N}_N \,\, \textnormal{falsely included in} \,\, \mathbf{C} \big) & = \mathbb{P}\big(\exists (n,h) \in \mathbf{N}_N \,\, \gamma(n,h)>q_{1-\alpha}  \big) \\
& \to \mathbb{P}(L^\rho>q_{1-\alpha} )  = \alpha.
\end{align*}
To establish this result, it suffices to show the following convergence 
\begin{align} \label{e:LrhoLz}
\max_{(n,h) \in \mathbf{N}_N} \gamma(n,h) = \max_{(n,h) \in \mathbf{N}_N} \tilde \gamma_0(n,h) \to L^\rho.
\end{align}
For the proof of this convergence, we again fix the outcome. Let $z \in (0,1)$ be a constant and consider the truncated triangle
\[
\blacktriangle_{0,1}^z := \{(x,y) \in \blacktriangle_{0,1}: z \le y \le 1\}
\]
and the corresponding set
\[
\mathbf{N}^{z}_N:=\{(n,h) \in \mathbf{N}_N: h \ge zN\}.
\]
Using the continuity of $W$ and the fact that $\inf_{y \ge z} \rho(y) $ is bounded away from $0$,  it follows that
\[
\max_{(n,h) \in \mathbf{N}^{z}_N} \tilde \gamma_0(n,h) \to \sup_{\blacktriangle_{0,1}^z} \frac{\|W(x+y)-2W(x)+W(x-y)\|}{\rho(y)}=:L^{\rho,z}
\]
as $N\to \infty$.
Now, $L^\rho$ is by definition 
\begin{align} \label{e:LLz}
   L^\rho= \lim_{z \downarrow 0} L^{\rho,z} ,
\end{align}
 and hence 
using a double limit $\lim_{z \downarrow 0} \lim_{N \to \infty}$ with the bounds
\[
\max_{(n,h) \in \mathbf{N}^{z}_N} \tilde \gamma_0(n,h)  \le \max_{(n,h) \in \mathbf{N}_N} \tilde \gamma_0(n,h)  \le L^\rho
\]
shows the desired convergence in eq. \eqref{e:LrhoLz}.

Weak localization in the case of Proposition \ref{prop2} follows the same arguments and the proof is therefore omitted. 
To show the strong localization property (in both propositions), one can now take the same steps as in the proof of Theorem \ref{theo2} (logarithmic case, now with $\beta=1/2$), and obtain the desired results. More precisely, we can focus on Step II in the proof of Theorem \ref{theo2} where we choose $r_{N,k}$ as described in eq. \eqref{e:prop:uN} for the logarithmic case with $\beta=1/2$. We need to show divergence of the $M_k$s defined in \eqref{def:Mk} and therefore define $H_k$ as in \eqref{e:def:Hk}. We then go through the calculation \eqref{e:logcalc} which proves this divergence.\\
Finally, we make a small comment on an assumption imposed for the Gaussian case, which is that $q_{1-\alpha}>C_W$ (the constant from Assumption \ref{ass_5}). Roughly speaking, this extra assumption is needed, because when going through the proof of the strong localization property, we had before the relation \eqref{e:nospur} for any null-sequence $u_N$. This relation followed from \eqref{e:g:s2} and the fact that $W$ was in the Hölder space. For the Gaussian case, discussed in this section, $W$ is not necessarily in the Hölder space (it would be if $C_W=0$). As a consequence, we get the modification of \eqref{e:nospur} that a.s.
\begin{align*} 
\limsup_N \max_{(n,h) \in \mathbf{N}_N: h <Nu_N}\gamma_0(n,h) \le C_W.
\end{align*}
 As $C_W<q_{1-\alpha}$, this means that for small scales, we will have (asymptotically)
\[
\max_{(n,h) \in\mathbf{N}_N: h <Nu_N}\gamma_0(n,h) <q_{1-\alpha}
\]
and again no false detections can take place.

\subsection{Proof of Theorem \ref{propmixing}} We will first prove  weak convergence of the process $\{ P_N(x,s)\}_{x,s \in [0,1]} $  in $\mathcal{C}([0,1] \times [0,1], \mathbb{R})$, the space of continuous functions, defined on the unit square and equipped with the respective sup-norm. For this purpose, it suffices to show convergence of the marginals and tightness (Theorem 16.5 in
\cite{kallenberg:2002}).
For convergence of the marginals, we take a total of $M$ tuples $(x_1,s_1),...,(x_M,s_M)$ and consider the vector 
\[
v_N := (P_N(x_m,s_m))_{m=1,...,M}.
\]
Convergence of $v_N$ now follows directly from Theorem 3.27 in \cite{dehling:mikosch:sorensen:2002}, with $\delta=1$ and using Assumption \ref{assmixing} parts i) (moment assumption), part ii) (fast decay of mixing coefficients) and part iii) (convergence of covariance matrix for $v_N$). Notice that we have used that by assumption $p>q\ge 4$. The latter inequality  follows because $q>1+1/r>2$ and $q$ is by assumption an even integer. \\
To prove  tightness, we introduce  the notation of the modulus of continuity
\[
\omega(f,h) := \sup_{x,y,s,t, \in [0,1]} \{|f(x,s)-f(y,t)|: d((x,s),(y,t))<h\} .
\]
Here, $d$ is a metric on the two-dimensional unit cube, defined as
\[
d((x,s),(y,t)) = \max(|x-y|, |s-t|^{r})
\]
(the parameter $r$ appears in first part of Assumption \ref{assmixing}).
The packing number w.r.t. $d$ and $h$ of the unit cube is defined as the smallest number of points $(x_1,s_1),...,(x_D,s_D)$, such that  
\[
\text{there exists no point }
(y,t) \in [0,1]^2 \text{ with }~ d((x_i,s_i),(y,t))>h \,\, \forall ~i=1,...,D.
\]
For the present case, it is easy to show \citep[with tools as on p. 98 in ][]{vaart:wellner:1996}  that the packing number with respect to the metric $d$ can be upper bounded by
\[
D = D_d(h) \le h^{-1} \times h^{-1/r}.
\]
We want to apply Theorem 2.2.4 in \cite{vaart:wellner:1996} and for this purpose also need a bound for $\mathbb{E}|P_N(x,s)-P_N(y,t)|^q$, which we derive with the help of Theorem 3 in \cite{yoshihara:1978}. To be precise, we use this result with the parameter choices $m=q$ and $\delta=p-q$ to obtain 
\[
\mathbb{E}|P_N(x,s)-P_N(y,t)|^q \le C d((x,s),(y,t)).
\]
Then we apply Theorem 2.2.4 in \cite{vaart:wellner:1996}, with $\psi(x)=x^{q}$. For the theorem to yield a meaningful bound, an integral over the packing numbers has to convergence, which is where we use the requirement $q>1+1/r$ made in our Assumption \ref{assmixing}. In  named theorem, the parameters $(\eta, \delta) =(\eta_N, \delta_N) $ need to be chosen depending on the sample size $N$, such that $(\eta_N, \delta_N) \to (0,0)$ with $\delta_N \eta_N^{-2/q(1+1/r)} \to 0$, which then yields 
\begin{align} \label{e:tightq}
\lim_{\delta \downarrow 0}\lim_{N \to \infty}\{\mathbb{E}|\omega(P_N,\delta)|^q\}^{1/q} =0.
\end{align}
Together, this equicontinuity and the previously established convergence of the marginals show the weak convergence of the  process $P_N$ in the space  $\mathcal{C}([0,1] \times [0,1], \mathbb{R})$ to a centered Gaussian limiting process. In other words, with $\mathbb{B}=\mathcal{C}([0,1],\mathbb{R})$ the process $\{P_N(x): x \in [0,1]\}$ converges to a Gaussian process $\{W(x): x \in [0,1]\}$ in $\mathcal{C}([0,1],\mathbb{B})$. 
\\
Now, we want to show that this convergence holds even w.r.t. the stronger topology of the Hölder space. It will suffice to show tightness of the sequence $(P_N)_{N \in \mathbb{N}}$ in the Hölder space $\mathcal{C}^\rho([0,1],\mathbb{B})$ and for this purpose  we apply Theorem 2.1 in \cite{rackauskas:suquet:2009}. This theorem has three conditions, which we state and validate in turn.
\begin{itemize}
    \item[(i)] For any $x$ the process $(P_N(x))_{N \in \mathbb{N}}$ is tight in the Banach space $\mathbb{B}$.
\end{itemize}
This follows from the weak convergence in first part of our proof.
\begin{itemize}
    \item[(ii)] It holds that
    \[
    \max_{1 \le n \le N} \frac{\|\varepsilon_n\|}{N^{1/2-\beta}} =o_P(1).
    \]
\end{itemize}
Using the union bound and Markov's inequality yields for any $\zeta>0$ 
\begin{align*}
    \PR \lb   N^{\beta - 1/2} \max_{1\leq i \leq N} \| \varepsilon_i \| > \zeta \rb
    \lesssim N^{p ( \beta - 1/2 ) } \sum_{i=1}^N \E \| \varepsilon_i \|^{p}
    \lesssim N^{p ( \beta - 1/2 ) +1 }
    = o(1), 
 \end{align*}
where the last equality follows from $p>(1/2-\beta)^{-1}$.

 The final condition (iii) in \cite{rackauskas:suquet:2009} is a kind of equicontinuity condition in Hölder spaces. It is based on the realization that Hölder norms are topologically equivalent to some (much easier) sequence norms of dyadic numbers, with the original result dating back to \cite{ciesielski:1960}. We do not state the condition here, but only a much simpler moment condition, that has been shown to be sufficient in \cite{kutta:doernemann:2025} (see the proof of Lemma 3.3 in this reference).
 \begin{itemize}
    \item[(iii)] For our parameter $q$ (which as we recall satisfies $q>(1/2-\beta)^{-1}$) and some fixed constant $C$ we have
    \[
    \mathbb{E}\Big\| \sum_{i=\ell}^k \varepsilon_i \Big\|^{q} \leq C |k-\ell|^{q/2}.
    \]
\end{itemize}
This statement follows  by similar but simpler arguments as our proof of
\eqref{e:tightq} \citep[applying Theorem 2.2.4 in][]{vaart:wellner:1996}, which is omitted for the sake of brevity. This completes the proof.

\subsection{Proof of Theorem \ref{theo3}} 
Before conducting the proof we want to remind the reader that the distribution of the errors $\varepsilon_{n,M}$ (defined in \eqref{e:errorF}) depends on both $n$ and $N$. Also recall that $\varepsilon_{n,M}$ depends on the number of points sampled, $M$, which is either a fixed number or depends on $N$ ($M=M_N \to \infty$ as $N \to \infty$).
The proof follows again by application of Theorem 2.1 in \cite{rackauskas:suquet:2009}. 
We check the three assumptions of named theorem, beginning with
\begin{itemize}
    \item[(ii)] It holds that 
    \begin{align} \label{e:hdn}
       \max_{1 \le n \le N} \frac{\|\varepsilon_{n,M}\|}{\log^\beta(N)} =o_P(1).
    \end{align}
\end{itemize}
It suffices to show that (uniformly in $N$)
\begin{align} \label{e:boundmom}
\sup_n \mathbb{E}\exp(\|\varepsilon_{n,M}\|^{1/\beta'})<\infty, 
\end{align}
for some $\beta'<\beta$. This is because Lemma 2.2.2 in \cite{vaart:wellner:1996} implies that (for all $N>1$)
\[
\mathbb{E}\bigg[\max_{1 \le n \le N} \frac{\|\varepsilon_{n,M}\|}{\log^\beta(N)}\bigg] \lesssim \frac{ \log^{\beta'}(N)}{\log^\beta(N)}=o(1).
\]
Thus, we consider the moments of the errors $\varepsilon_i$ with $\beta'=1/2$. Using the fact that that 
 $\mathbb{E}[R] = \int_0^\infty (1-F_R(x))dx$ for any 
 non-negative  random variable $R$ with existing expectation and cdf $F_R$, we obtain
\begin{align} \label{e:expe}
    \mathbb{E} \big[ \exp(\|\varepsilon_{n,M}\|^2) \big]  = & \int_0^\infty \mathbb{P}\big(\exp(\|\varepsilon_{n,M}\|^2)>t \big) dt =  \int_0^\infty \mathbb{P}\big(\|\varepsilon_{n,M}\|>\log^{1/2}(t) \big) dt.
\end{align}
Recalling that in this proof $\|\cdot\|$ denotes the weighted $L^2$-norm for functions on the real line, it can be upper bounded by the sup-norm
\[
\bigg(\int_\mathbb{R} f^2(x) w(x) dx\bigg)^{1/2} \le \sup_{x \in \mathbb{R}} |f(x)| \bigg(\int_\mathbb{R} w(x) dx\bigg)^{1/2} =  \sup_{x \in \mathbb{R}} |f(x)|,
\]
where we have used that $w$ is a probability density. Next, recall the two-sided Dvoretzky–Kiefer–Wolfowitz inequality (see \cite{massart:1990}), which implies
for any $a>0$ that
\[
\mathbb{P}\Big(\sup_x |\widehat{F}_{n,M} (x) -F_n(x) |>a \Big) \le 2 \exp(-2Ma^2)
\]
where, as we recall, $M$ is the sample size in the construction of the empirical cdf. Now, with these identities and 
the definition $\varepsilon_{n,M} = \sqrt{M}[\widehat{F}_{n,M}-F_n]$ we get for the right side of  \eqref{e:expe} the upper bound
\begin{align*}
   \int_0^\infty \mathbb{P}\big(\sup_x |\widehat{F}_{n,M}-F_n|>\log^{1/2}(t)/\sqrt{M} \big) dt  &  \le 1+ 2 \int_1^\infty \exp(-2 \log(t)) dt \\
   & = 1 + 2 \int_1^\infty t^{-2} dt <\infty.  
\end{align*}
This shows the desired bound for the square exponential moment of $\varepsilon_{n,M}$ in \eqref{e:boundmom} and we notice that it is independent of $n$. We turn to the next condition.
\begin{itemize}
    \item[(i)] For any $x$ the partial sum process $(P_N(x))_{N \in \mathbb{N}}$ is tight in the Hilbert space of square integrable functions. 
\end{itemize}
W.l.o.g. we may assume that $x \le \theta_1$and hence reduce the problem to i.i.d. data (the noise terms $\varepsilon_{i,M}$ with $1 \le i \le c_1$ are i.i.d.). If $x> \theta_1$, say $x \in (\theta_1, \theta_2]$ one can divide up the partial sum process $P_N(x) = [P_N(x)-P_N(\theta_1)] + [P_N(\theta_1)]$ and show convergence of both (stochastically independent) terms in the squared brackets separately, which again reduces the problem to studying sums of i.i.d. random variables.\footnote{There is a minor technicality that we neglect here. Due to linear interpolation at the boundary, $[P_N(x)-P_N(\theta_1)]$ and $[P_N(\theta_1)]$ are only independent if $\theta_1 N $ is a natural number. If this is not the case, $[P_N(x)-P_N(\theta_1)]$ and $[P_N(\theta_1)]$ are independent after removing a negligible term.} So, let us fix $0 <x \le \theta_1$. Then we have 
\begin{align} \label{e:PNSx}
    P_N(x) = & N^{-1/2} \sum_{i=1}^{\lfloor x N \rfloor} \varepsilon_{i,M}+N^{-1/2} \big(Nx-\lfloor N x \rfloor\big) \varepsilon_{\lfloor N x \rfloor+1,M}\\
    =  &\bigg(\frac{\lfloor x N \rfloor^{1/2}}{N^{1/2}}\bigg) \bigg\{\lfloor x N \rfloor^{-1/2} \sum_{i=1}^{\lfloor x N \rfloor} \varepsilon_{i,M}\bigg\}+o_P(1). \nonumber
\end{align}
The first equality is the definition of $P_N$. The remainder in the second one follows because by the previous part of this proof
\[
\max_{1 \le n \le N} \frac{\|\varepsilon_{n,M}\|}{\sqrt{N}} =o_P(1).
\]
It will now suffice to show convergence of the part in the curly brackets
\[
S_x:= \lfloor x N \rfloor^{-1/2} \sum_{i=1}^{\lfloor x N \rfloor} \varepsilon_{i,M}
\]
and we can apply a central limit theorem from \cite{samur:1984}, Corollary 4.6. The conditions of this corollary are easily checked -  it is a CLT for triangular arrays of mixing random variables with stationary rows. Our data are stationary and mixing (because they are even i.i.d.). Assumptions (1)-(2) of the corollary follow because our noise terms $\varepsilon_{n,M}$ are centered and have uniformly bounded fourth moments (by the first part of this proof). The same is true for the first statement of condition (3), and we only need to show that the asymptotic covariance operator exists. Here
\[
\mathbb{E}\langle S_x, f\rangle \langle S_x,g\rangle = \int \int f(s) g(t) \mathbb{E}[\varepsilon_{1,M}(s)\varepsilon_{1,M}(t)] w(s) w(t) ds dt.
\]
Recalling the definition of $\varepsilon_{n,M}$ in \eqref{e:errorF} we obtain that 
\begin{align} \label{e:conunif}
\mathbb{E}[\varepsilon_{1,M}(s)\varepsilon_{1,M}(t)] & = F^{(1)}_N(\min(s,t))-F_N^{(1)}(s)F_N^{(1)}(t) \\
& \to F^{(1,*)}(\min(s,t))-F^{(1,*)}(s)F^{(1,*)}(t)=:c^*(s,t),\nonumber
\end{align}
where the convergence is uniform and follows from Assumption \ref{ass_3}. 
Now, plugging in the definition of $c^*(s,t)$, we get
\begin{align*}
 & \big|\mathbb{E}\langle S_x, f\rangle \langle S_x,g\rangle - C^\varepsilon(x,x)[f,g]\big|\\
= & \bigg|\int \int \big(f(s) g(t) \big)\Big(\mathbb{E}[\varepsilon_{1,M}(s)\varepsilon_{1,M}(t)] -c^*(s,t)\Big)  w(s) w(t) ds dt \bigg|.
\end{align*}
The right side convergences to $0$ using the uniform convergence in eq. \eqref{e:conunif} and
we thus get
\begin{align} \label{e:covSx}
\mathbb{E}\langle S_x, f\rangle \langle S_x,g\rangle \to \int \int f(s) g(t) c^*(s,t) w(s) w(t) ds .  
\end{align}
The final condition (4) of the cited Corollary 4.6 in \cite{samur:1984} requires that 
\begin{align} \label{e:epsmom}
    \lim_{d \to \infty}\sup_N\mathbb{E}\|\varepsilon_{1,M}-\Pi_d[\varepsilon_{1,M}]\|^2=0.
\end{align}
Here, the supremum over $N$ refers to the fact that $\varepsilon_{1,M}$ has a distribution also implicitly depending on $N$, through the function $F_N^{(1)}$. The object $\Pi_d$ is a projection on a $d$-dimensional subspace of $L^2(\mathbb{R})$ that we may choose.
If $M$ is fixed, the proof is easier and we focus on the case where $M=M_N \to \infty$. 
The key step is proving weak convergence of $\varepsilon_{1,M}$ to $B(F^{(1,*)}(\cdot))$, where $B$ is a Brownian bridge. 
$B(F^{(1,*)}(\cdot))$ is an element of $L^2(\mathbb{R})$ as we have defined it (because it has continuous, bounded sample paths). 
The convergence now can be proved in different ways, but it follows readily by a KMT approximation. Here, consider an appropriate probability space and i.i.d. uniformly distributed random variables on the unit interval $U_1, U_2,...$. On this probability space, a version of $\varepsilon_{1,M}$ is given by $\varepsilon_{1,M} = \sqrt{M}[\widehat{F}-F^{(1)}]$ with
\[
\widehat{F} (t):= \frac{1}{M}\sum_{m=1}^M \mathbb{I}\{U_m\le F_N^{(1)}(t)\}.
\]
It can be shown using the KMT approximations (\cite{kmt:1975}; Theorem 4) that we can construct a Brownian bridge $B$ s.t.
\[
\mathbb{P}\Big(\sup_{t \in [0,1]} |\frac{1}{\sqrt{M}}[\sum_{m=1}^M \mathbb{I}\{U_m \le t\}-t] - B(t)|>CM^{-1/4}\Big)\le CM^{-1/4},
\]
for some large enough constant $C$. Substituting  $t$ by $F_N^{(1)}(t)$ in the above equation, we get that 
\[
\sup_t|\varepsilon_{1,M}(t) - B(F_N^{(1)}(t))|=o_P(1).
\]
In turn, using uniform convergence of $F_N^{(1)}$ to $F^{(1,*)}$ we get that
\[
\sup_t|B(F_N^{(1)}(t)) - B(F^{(1,*)}(t))|=o_P(1)
\]
as $N \to \infty$.
For any fixed $d$ and as $N,M_N \to \infty$ we then get w.r.t. the (weaker) $L^2$ topology
\[
\Big(\varepsilon_{1,M}-\Pi_d[\varepsilon_{1,M}] \Big)\overset{\mathbb{P}}{\to} \Big(B(F^{(1,*)})-\Pi_d[B(F^{(1,*)}]\Big).
\]
Since all moments of $\varepsilon_{1,M}$ are bounded, it follows by the above convergence that also as $N, M_N \to \infty$ 
\[
\mathbb{E} \Big[ \Big(\varepsilon_{1,M}-\Pi_d[\varepsilon_{1,M}] \Big)^2 \Big]  \to \mathbb{E} \Big [ \Big(B(F^{(1,*)})-\Pi_d[B(F^{(1,*)})]\Big) \Big ] ^2.
\]
 The right side is not dependent on $N$ anymore, and choosing a sequence of increasing vector spaces $V_d$ of dimension $d$ s.t. $\bigcup_d V_d$ is dense in $L^2(\mathbb{R})$ and defining $\Pi_d$ as the projection on $V_d$ implies that
 \[
 \lim_{d \to \infty}\sup_N\mathbb{E}\Big [ \|\varepsilon_{1,M}-\Pi_d[\varepsilon_{1,M}]\|^2\Big ] =0.
 \]
We have shown all assumptions of Corollary 4.6 in \cite{samur:1984} which then implies that the process $ \{ S_x \}_x$ converges to a centered Gaussian $G$ in $L^2(\mathbb{R})$, characterized by the covariance in eq. \eqref{e:covSx}. Recalling \eqref{e:PNSx} we obtain that $P_N(x) $
converges to the Gaussian law characterized by 
 the covariance function $C^\varepsilon$ is defined in Theorem \ref{theo3}. We turn to the last condition    in \cite{rackauskas:suquet:2009} and therefore  define the quantities
    \begin{align}
        t_k & = t_{j,k} = k 2^{-j}, \quad 0 \leq k \leq 2^j, ~j\geq 1 
        \end{align}
    In the following denote by $\log$ the logarithm with basis $2$.
\begin{itemize}
    \item[(iii)] It holds that for any $a>0$
    \[
    \lim_{J \to \infty} \limsup_N \mathbb{P}\Big(\max_{J \le j \le \log(N)} \max_k\|S^{\lfloor 
 Nt_{k+1} \rfloor}_{\lfloor Nt_k\rfloor+1}\|> aN^{1/2}\rho(2^{-j})\Big) =0.
    \]
\end{itemize}
We slightly facilitate the proof by taking the maximum only over such $k$ where $ \lfloor Nt_{k+1}\rfloor \le c_1$. 
Using the definition of the logarithmic weight function $\rho(x) = x^{1/2}\log^{\beta}(x^{-1})$ we have
\begin{align*}
    & \mathbb{P}\Big(\max_{J \le j \le \log(N)} \max_k\|S^{\lfloor 
 Nt_{k+1} \rfloor}_{\lfloor Nt_k\rfloor+1}\|/(N^{1/2}2^{-j/2})> a \log^{\beta}(2^{j})\Big) \\
    \le & \sum_{j \ge J} \sum_{k=0}^{2^j} \mathbb{P}\Big(\|S^{\lfloor 
 Nt_{k+1} \rfloor}_{\lfloor Nt_k\rfloor+1}\|/\sqrt{\lfloor 
 Nt_{k+1} \rfloor -\lfloor Nt_k\rfloor}> C_{S,1} a \log^{\beta}(2^{j})\Big)\\
    \le & \sum_{j \ge J} \sum_{k=0}^{2^j} \exp(-C_{S,2} \log^{2\beta}(2^j)) \le  \sum_{j \ge J} (2^j+1) \exp(-C_{S,2}  \log^{2\beta}(2^j)), 
\end{align*}
where   $C_{S,1}$ and $C_{S,2}$ are positive and sufficiently large constants  and do not  depend on $N$.
We explain the second inequality of the above derivations below. First, we appreciate that, if the above bound holds, it follows that the series on right (which does not depend on $N$) is converging because $2\beta>1$. Hence, the whole object converges  to $0$ if we let $J \to \infty$.\\
Now, to justify the second inequality. Let $L$ be a generic natural number and we may consider  $S_1^{L}$ due to stationarity of the data (before the first change). We then rewrite 
\begin{align*}
    S_1^{L}/\sqrt{L} = \frac{\sqrt{M}}{\sqrt{L}}\sum_{n=1}^{L} (\widehat{ F}_{n,M}-F_n) = \sqrt{ML}(\widetilde{F}_n-F^{(1)}).
\end{align*}
Here, we have used that $F_n = F^{(1)}_N$ (since we are before the first change) and we have defined $\widetilde{F}_n := \frac{1}{L} \sum_{k=1}^L \widehat{F}_{k,M}$. $\widetilde{F}_n$ is again an empirical cdf (now over $ML$ data points) and we get square-exponential concentration around $F^{(1)}_N$ as in the first part of this proof.
\small\renewcommand{\baselinestretch}{0.95}
\putbib
\end{bibunit}
\end{document}